\documentclass[10pt,letterpaper]{amsart}

\usepackage{amsmath,amsthm,amsfonts,amssymb}
\usepackage{mathtools} 
\usepackage{mathrsfs} 
\usepackage[pdftex,bookmarks=true]{hyperref} 
\usepackage{amsrefs}

\newcommand{\mfr}[1]{\ensuremath \mathfrak{#1}}
\newcommand{\mbb}[1]{\ensuremath \mathbb{#1}}
\newcommand{\mbf}[1]{\ensuremath \mathbf{#1}}
\newcommand{\mcl}[1]{\ensuremath \mathcal{#1}}
\newcommand{\msc}[1]{\ensuremath \mathscr{#1}}
\newcommand{\mrm}[1]{\ensuremath \mathrm{#1}}

\newcommand{\R}{\mathbb{R}}

\DeclareMathOperator{\SL}{SL}
\DeclareMathOperator{\SO}{SO}
\DeclareMathOperator{\Rc}{Rc}
\DeclareMathOperator{\id}{id}
\DeclareMathOperator{\tr}{tr}
\DeclareMathOperator{\Hess}{Hess}
\newcommand{\smfrac}[2]{{\textstyle{\frac{#1}{#2}}}}
\newcommand{\half}{{\smfrac{1}{2}}}

\numberwithin{equation}{section}

\theoremstyle{plain} 
\newtheorem{thm}[equation]{Theorem}

\newtheorem{lem}[equation]{Lemma}
\newtheorem{prop}[equation]{Proposition}

\theoremstyle{definition}

\theoremstyle{remark}
\newtheorem{rem}[equation]{Remark}

\newtheorem*{ack}{Acknowledgement}

\begin{document}

\author[M.~B.~Williams]{Michael Bradford Williams}
\address{Department of Mathematics, The University of California, Los Angeles}
\email{mwilliams@math.ucla.edu}
\urladdr{http://math.ucla.edu/~mwilliams}

\title{Stability of solutions of certain extended Ricci flow systems}
\date{}

\begin{abstract} 
We consider four extended Ricci flow systems---that is, Ricci flow coupled with other geometric flows---and prove dynamical stability of certain classes of stationary solutions of these flows.  The systems include Ricci flow coupled with harmonic map flow (studied abstractly and in the context of Ricci flow on warped products), Ricci flow coupled with both harmonic map flow and Yang-Mills flow, and Ricci flow coupled with heat flow for the torsion of a metric-compatible connection.  The methods used to prove stability follow a program outlined by Guenther, Isenberg, and Knopf, which uses maximal regularity theory for quasilinear parabolic systems and a result of Simonett.
\end{abstract}



\maketitle
    
\section{Introduction}

A fundamental problem in the study of differential equations is to determine the asymptotic behavior of solutions of a given equation.  This problem is central to the application of differential equations to Riemannian geometry.  For example, Eells and Sampson demonstrated the existence of harmonic maps by proving that solutions of the harmonic map flow converge \cite{EellsSampson1964}.  Hamilton placed strong restrictions on the topology of three-manifolds admitting positive Ricci curvature by proving that solutions of Ricci flow converge to space forms \cite{Hamilton1982} (see also \cites{Hamilton1986,BohmWilking2008}).

One way to phrase this problem is in terms of \textit{stability} of stationary solutions of the system of equations in question:~do solutions with initial data near a fixed point converge to that fixed point?  For Ricci flow, whose fixed points include Einstein and Ricci-flat metrics, there are many stability results.  

To mention a few of these results in the compact case, Ye proved that Einstein metrics with certain curvature pinching properties are stable \cite{Ye1993}; Guenther, Isenberg and Knopf proved that certain flat and Ricci-flat metrics are stable \cite{GuentherIsenbergKnopf2002}, and some of these results were improved by \v{S}e\v{s}um \cite{Sesum2006}; using the results of \v{S}e\v{s}um, Dai, Wang, and Wei proved that K\"{a}hler-Einstein metrics with non-positive scalar curvature are stable \cite{DaiWangWei2007}; Knopf and Young proved that hyperbolic space forms are stable \cite{KnopfYoung2009}; Wu proved that compact quotients of complex hyperbolic spaces are stable \cite{Wu2013}.  We should note that these authors use various techniques and obtain stability relative to various topologies on the space of metrics.

The purpose of this paper is to describe the stability of solutions of certain extended Ricci flow systems, which arise in various geometric contexts involving Riemannian manifolds with additional structure, and to show that techniques for proving stability of Ricci flow solutions (due to Guenther, Isenberg, and Knopf \cite{GuentherIsenbergKnopf2002}) apply to a wider range of flows.  First, we consider Ricci flow coupled with the harmonic map flow; see Section \ref{sec:hrf}.  If $\phi : (\mcl{B},g) \rightarrow (\mcl{N},\gamma)$ is a map of Riemannian manifolds, \textit{harmonic-Ricci flow} is the coupled system
\begin{equation}\label{eq:hrf}
\begin{aligned}
\partial_t g    &= -2 \Rc + 2c \, \nabla \phi \otimes \nabla \phi \\
\partial_t \phi &= \tau_{g,\gamma} \phi 
\end{aligned}
\end{equation}
where $\tau_{g,\gamma} \phi$ is the harmonic map Laplacian of $\phi$, $\nabla\phi \otimes \nabla\phi = \phi^* \gamma$, and $c$ is a (possibly time-dependent) coupling constant.  This was introduced by List in the case $\mcl{N} = \R$ \cite{List2008}, and the general case was addressed by M\"uller \cite{Muller2012}.  This flow appears naturally in certain situations; see Section \ref{subsec:hrf-setup} for examples.

We will demonstrate the stability of fixed points $(g,\phi)$ of a modified version of \eqref{eq:hrf}, where $\phi$ is constant and $g$ is a strictly linearly stable Einstein metric.  For this and the other flows in this paper, we assume linear stability of a metric with respect to the following linear operator on symmetric $2$-tensors on $(\mcl{B},g)$:  
\begin{equation}\label{eq:L0}
\mbf{L}_{0} h  := \Delta_{\ell} h + 2\lambda h 
\end{equation}
where $\Delta_\ell$ is the Lichnerowicz Laplacian and $\lambda < 0$ is specified (usually as the Einstein constant of a metric).  Linear stability means that there exists $\epsilon > 0$ such that, for all symmetric 2-tensors $h$ taken from some appropriate tensor space, 
\begin{equation}\label{eq:lin-stab}
(\mbf{L}_0 h,h) \leq -\epsilon \|h\|^2,
\end{equation}
relative to the $L^2$ inner product induced by $g$.  

\begin{thm}\label{thm:hrf-stab}
Let $\phi : (\mcl{B}^n,g) \rightarrow (\mcl{N},\gamma)$ be a map of Riemannian manifolds.  Suppose that $\mcl{M}$ is compact and orientable, $g$ is a strictly linearly stable Einstein metric, and $\phi$ is constant.  Then for any $\rho \in (0,1)$, there exists $\theta \in (\rho,1)$ such that the following holds.

There exists a $(1+\theta)$-little-H\"{o}lder neighborhood $\mcl{U}$ of $(g,\phi)$ such that for all initial data $\big(\widetilde{g}(0),\widetilde{\phi}(0)\big) \in \mcl{U}$, the unique solution $\big(\widetilde{g}(t),\widetilde{\phi}(t)\big)$ of curvature-normalized harmonic-Ricci-DeTurck flow (\ref{eq:cnhrdtf}) exists for all $t \geq 0$ and converges exponentially fast in the $(2+\rho)$-H\"{o}lder norm to $(g,\phi_{\infty})$, where $\phi_{\infty}$ is constant.
\end{thm}

Second, we consider Ricci flow on warped products; see Section \ref{sec:wp}.  Warped products of the form $\R \times \mcl{F}^m$, where $\mcl{F}^m$ is a positively-curved Einstein manifold, were used by Simon to construct metrics with pinching singularities \cite{Simon2000}.  More recently, Lott and \v{S}e\v{s}um studied Ricci flow on compact warped products $\mcl{M}^3 = \mcl{B}^2 \times \mcl{S}^1$ and proved several stability-type results \cite{LottSesum2011}.  Tran also considered Ricci flow on warped products $\mcl{B}^n \times \mcl{F}^m$ where $\mcl{F}$ is Ricci-flat \cite{Tran2012}.  

We will study Ricci flow on multiply-warped products.  Specifically, fix an integer $m>0$ and let $(\mcl{B},g)$ and $(\mcl{F}_\alpha,g_\alpha)$ be Riemannian manifolds with $\dim \mcl{B} = n$ and $\dim \mcl{F}_\alpha = n_\alpha$, and let $\phi_\alpha \in C^\infty(\mcl{B})$, where $\alpha = 1,\dots,m$.  Define the following product manifold and multiply-warped product metric,
\[ \mcl{M} := \mcl{B} \times \prod_{\alpha=1}^m \mcl{F}_\alpha, \qquad
   \mbf{g} := g + \sum_{\alpha=1}^m e^{2\phi_\alpha} g_\alpha. \]
We will assume that each fiber factor is $\mu_\alpha$-Einstein.  On $(\mcl{M},\mbf{g})$, Ricci flow is the coupled system
\begin{equation}\label{eq:wprf}
\begin{aligned}
\partial_t g      &= -2 \Rc + 2 \sum_\alpha n_\alpha \, d\phi_\alpha \otimes d\phi_\alpha \\
\partial_t \phi_\alpha &= \Delta \phi_\alpha - \mu_\alpha e^{-2\phi_\alpha} \qquad \alpha=1,\dots,m
\end{aligned}
\end{equation}
which can be interpreted as a modified version of \eqref{eq:hrf} where the target is $\mbb{R}^m$.  Not surprisingly, the behavior of solutions depends strongly on the signs of $\mu_\alpha$.  When $\mu_\alpha \leq 0$, we obtain a result similar to Theorem \ref{thm:hrf-stab} above.

\begin{thm}\label{thm:wprf-stab}
Let $\mbf{g} = g +\sum e^{2\phi_\alpha} g_\alpha$ be a multiply-warped product metric on $\mcl{M} = \mcl{B} \times \prod \mcl{F}_\alpha$, where $\mathcal{B}^n$ is compact and orientable and $(\mcl{F}_\alpha^{n_\alpha},g_\alpha)$ is $\mu_\alpha$-Einstein.  Suppose that $g$ is a strictly linearly stable $\lambda$-Einstein metric, each $\phi_\alpha$ is constant, and each $\mu_\alpha$ equals either $\lambda$ or $0$.  Then for any $\rho\in(0,1)$, there exists $\theta\in(\rho,1)$ such that the following holds.

There exists a $(1+\theta)$-little-H\"{o}lder neighborhood $\mcl{U}$ of $\mbf{g}$ such that for all initial data $\widetilde{\mbf{g}}(0)\in \mcl{U}$, the unique solution $\widetilde{\mbf{g}}(t)$ of curvature-normalized multiply-warped product Ricci-DeTurck flow (\ref{eq:cnmwprdtf}) exists for all $t \geq 0$ and converges exponentially fast in the $(2+\rho)$-H\"{o}lder norm to $\mbf{g}_\infty := g + \sum e^{2(\phi_\alpha)_\infty} g_\alpha$, where $(\phi_\alpha)_\infty$ is constant.  In particular, if $\mu_\alpha = \lambda$, then $(\phi_\alpha)_\infty = \phi_\alpha$.
\end{thm}

Next, we consider \textit{locally $\R^N$-invariant Ricci flow}; see Section \ref{sec:irf}.  This flow was introduced by Lott as a means to prove that if $(\mcl{M}^3,g(t))$ is a compact, Type III Ricci flow solution with diameter $O(\sqrt{t})$, then the pull-back solution $(\widetilde{\mcl{M}},\widetilde{g}(t))$ on the universal cover converges to a homogeneous Ricci soliton \cite{Lott2010}.  For this, Lott considered a class of ``twisted'' principal $\R^N$-bundles.  Certain metrics on such bundles $\R^N \hookrightarrow \mcl{M}^{n+N} \rightarrow \mcl{B}^n$ can be represented locally as a triple $\mbf{g} = (g,A,G)$, where $g$ is a metric on the base, $A$ is an $\R^N$-valued $1$-form corresponding to a connection on $\mcl{M}$, and $G$ is an inner product on the fibers.  Ricci flow on these \textit{locally $\R^N$-invariant} metrics decomposes into a Ricci flow-type equation for $g$, a Yang-Mills flow-type equation for $A$, and a heat-type equation for $G$:
\begin{equation}\label{eq:irf}
\begin{aligned}
\partial_t g &= -2 \Rc + \frac{1}{2} \nabla G \otimes \nabla G + dA \otimes dA \\
\partial_t A &= -\delta dA + \langle \nabla G,dA \rangle \\
\partial_t G &= \tau_{g,\theta} G - \frac{1}{2} dA \otimes dA
\end{aligned}
\end{equation}
Notation will be explained fully below.  An important ingredient in the proof of Lott's theorem is a set of stability results for this system, proved by Knopf in the cases $N=1$ or $n=1$ \cite{Knopf2009}.  We extend some of those results to arbitrary dimensions.

\begin{thm}\label{thm:irf-stab}
Let $\mbf{g}=(g,A,G)$ be a locally $\R^N$-in\-var\-i\-ant metric of the form (\ref{eq:the-metric}) on a product $\R^N \times \mcl{B}^n$, where $\mcl{B}$ is compact and orientable.  Suppose that $g$ is a strictly linearly stable $\lambda$-Einstein metric, $A$ vanishes, and $G$ is constant.  Then for any $\rho \in (0,1)$, there exists $\theta \in (\rho,1)$ such that the following holds.

There exists a $(1+\theta)$-little-H\"{o}lder neighborhood $\mcl{U}$ of $\mbf{g}$ such that for all initial data $\widetilde{\mbf{g}}(0) \in \mcl{U}$, the unique solution $\widetilde{\mbf{g}}(t)$ of curvature-normalized locally $\R^N$-invariant Ricci-DeTurck flow (\ref{eq:cnirdtf}) exists for all $t \geq 0$ and converges exponentially fast in the $(2+\rho)$-H\"{o}lder norm to $\mbf{g}_{\infty} := (g,A_{\infty},G_{\infty})$, where $A_{\infty}$ vanishes and $G_{\infty}$ is constant.
\end{thm}

Finally, we consider \textit{connection Ricci flow}; see Section \ref{sec:crf}.  This flow was introduced by Streets as a geometric interpretation of renormalization group flow on $(\mcl{B}^n,g)$ with $B$-field included, and which takes the form of Ricci flow coupled with heat flow for a closed three-form \cite{Streets2008} (see also \cites{Strominger1986,OliynykEtAl2006}).  Here, $n \geq 3$.  One can interpret the $B$-field strength as the torsion $\tau$ of a metric compatible connection, and Ricci flow in this setting becomes Ricci flow for $g$ coupled with heat flow for the torsion:
\begin{equation}\label{eq:crf}
\begin{aligned}
\partial_t g    &= -2 \Rc + \frac{1}{2} \mcl{H} \\
\partial_t \tau &= \Delta \tau
\end{aligned}
\end{equation}
where $\mcl{H}_{ij} = g^{pq} g^{rs} \tau_{ipr} \tau_{jqs}$.  There are also certain other assumptions on $\tau$ that we will make precise in Section \ref{subsec:crf-setup}.

We show that the flow is stable when the metric $g$ is Einstein and linearly stable, and the connection is the Levi-Civita connection of $g$.

\begin{thm}\label{thm:crf-stab}
Suppose that $(\mcl{B}^n,g)$ is compact and orientable, with $n \geq 3$, and that $g$ is a strictly linearly stable $\lambda$-Einstein metric.  Let $\nabla$ be a metric-compatible connection on $\mcl{B}$ with torsion $\tau=0$ (so $\nabla$ is the Levi-Civita connection of $g$).  Then for any $\rho \in (0,1)$, there exists $\theta \in (\rho,1)$ such that the following holds.

There exists a $(1+\theta)$-little-H\"{o}lder neighborhood $\mcl{U}$ of $(g,\tau)$ such that for all initial data $\big(\widetilde{g}(0),\widetilde{\tau}(0)\big) \in \mcl{U}$, the unique solution $\big(\widetilde{g}(t),\widetilde{\tau}(t)\big)$ of curvature-normalized connection Ricci-DeTurck flow (\ref{eq:cncrdtf}) exists for all $t \geq 0$ and converges exponentially fast in the $(2+\rho)$-H\"{o}lder norm to $(g,\tau)$.
\end{thm}

The proofs of these theorems follow the same general outline.  

\begin{enumerate}
\item Modify the flow so that the fixed points are more easily studied.  This involves rescaling and pulling back by diffeomorphisms (including a DeTurck trick), and the fixed points include Einstein metrics together with other objects relevant to the flow in question (e.g., maps or connections).   
\item Compute the linearization of the modified flow, and discuss linear stability at the fixed points.  We assume metrics are strictly linearly stable, but other objects might only be weakly linearly stable.  In that case we must understand the null eigenspaces.
\item Obtain dynamical stability by setting up the appropriate little-H\"older spaces and applying a theorem of Simonett.  (See Appendix \ref{app:stab} for the statement of the theorem.)
\end{enumerate}
This technique was introduced by Guenther, Isenberg, and Knopf \cite{GuentherIsenbergKnopf2002}, and has subsequently been used to prove several other results.  For example, as mentioned before, Knopf proves stability of certain solutions of locally invariant $\R^N$-invariant Ricci flow \cites{Knopf2009}.  Young considers Ricci flow coupled with Yang-Mills flow and proves stability of certain solutions \cite{Young2010}.  Wu, as cited above, uses these methods to prove that quotients of complex hyperbolic space are stable under Ricci flow \cite{Wu2013}.  The method has been adapted for use in the non-compact setting by Wu and the author \cites{Wu2013,WilliamsWu2013-dynamical}.

Step (3) in this technique relies on the maximal regularity theory of Da Prato and Grisvard \cite{DaPratoGrisvard1979}, which exploits the smoothing properties of quasilinear parabolic operators.  The actual dynamical stability then follows from a (quite general) theorem of Simonett, which is based on this maximal regularity theory \cite{Simonett1995}.  

\begin{rem}
Simonett's theorem has the feature of giving dynamical stability even in the presence of center manifolds, which occur when a fixed point is only weakly linearly stable, that is, when there is a non-trivial null eigenspace of the linearized flow operator.  The first analysis of center manifolds in problems relating to Ricci flow appeared in \cite{GuentherIsenbergKnopf2002}.  To allow for simplified statements of our theorems, we assume that the metrics are strictly linearly stable, and only consider weak stability the other objects (maps/connections, etc.).  The theorems could be adapted to account for center manifolds arising from weakly stable metrics, but this would require more knowledge of the null eigenspaces, and so the theorems would be more complicated to state.
\end{rem}

\begin{rem} 
The first three theorems here are true when the Einstein manifold $(\mcl{B},g)$ is replaced by a two-dimensional sphere with constant positive sectional curvature.  Unfortunately, the techniques used here do not generalize to higher dimensions for positively curved manifolds.  See \cite{Knopf2009}*{Remark 2}.
\end{rem}

\begin{ack}
The author wishes to thank Dan Knopf for his helpful comments and suggestions, and Peter Petersen for many enlightening discussions.
\end{ack}

\section{Harmonic-Ricci flow}\label{sec:hrf}
  
\subsection{Setup and examples}\label{subsec:hrf-setup}

Let us provide background for the coupled flow \eqref{eq:hrf}.  Let $(\mcl{B},g)$ be a closed Riemannian manifold, with $(\mcl{N},\gamma)$ a closed target manifold.  Let $\phi : \mcl{B} \rightarrow \mcl{N}$ be a smooth map.  The Levi-Civita covariant derivative $\nabla^{T\mcl{N}}$ of the metric $\gamma$ on $N$ induces a covariant derivative $\nabla^{\phi^* T\mcl{N}}$ on the pull-back bundle $\phi^* T\mcl{N} \rightarrow \mcl{M}$, given by
\[ \nabla^{\phi^* T\mcl{N}}_X Y = \nabla^{T\mcl{N}}_{\phi_* X} Y, \]
for $X \in C^\infty(T \mcl{M})$ and $Y \in C^\infty(T \mcl{N})$.  This and the Levi-Civita covariant derivative $\nabla^{T\mcl{B}}$ of the metric $g$ on $\mcl{B}$ induce covariant derivatives on all tensor bundles over $\mcl{B}$ of the form
\[ (T^*\mcl{B})^{\otimes p} \otimes (T\mcl{B})^{\otimes q} \otimes (\phi^* T^*\mcl{N})^{\otimes r} \otimes (\phi^* T\mcl{N})^{\otimes s}. \]
We refer to them simply as $\nabla$.  Related quantities are decorated with the metric name, if necessary, e.g., ${}^g \nabla$.  In local coordinates $(x^i)$ on $\mcl{B}$ and $(y^\alpha)$ on $\mcl{N}$,
\[ \nabla \phi = \phi_* = \partial_i \phi^{\alpha} \, dx^i \otimes \partial_{\alpha}|_{\phi} \in C^\infty(T^*\mcl{B} \otimes \phi^* T\mcl{N}). \]
Similarly, we have
\begin{align*}
\nabla^2 \phi 
&= \big( \partial_i \partial_j \phi^\alpha - {}^g \Gamma_{ij}^k \partial_k \phi^\alpha + ({}^\gamma \Gamma \circ \phi)_{\mu \nu}^\alpha \partial_i \phi^\mu \partial_j \phi^\nu \big) \, dx^i \otimes dx^j \otimes \partial_\alpha|_\phi \\
&\in C^\infty(T^*\mcl{B} \otimes T^*\mcl{B} \otimes \phi^* T\mcl{N}). 
\end{align*}
The \textit{harmonic map Laplacian} (or \textit{tension field}) of $\phi$ with respect to $g$ and $\gamma$ is
\begin{equation}\label{eq:tension}
\begin{aligned}
\tau_{g,\gamma} \phi 
&:= \tr_g \nabla^2 \phi \\
&= g^{ij} \big( \partial_i \partial_j \phi^\alpha - {}^g \Gamma_{ij}^k \partial_k \phi^\alpha + ({}^\gamma \Gamma \circ \phi)_{\mu \nu}^\alpha \partial_i \phi^\mu \partial_j \phi^\nu \big) \, \partial_\alpha|_\phi \\
&\in C^{\infty}(\phi^* T\mcl{N}).
\end{aligned}
\end{equation}
Additionally, the induced metric $\phi^* \gamma$ on $\phi^* T\mcl{N}$ can be expressed as
\begin{equation}\label{eq:phi-quad}
\nabla \phi \otimes \nabla \phi = \gamma_{\alpha \mu} \partial_i \phi^\alpha \partial_j \phi^\mu \, dx^i \otimes dx^j = \phi^* \gamma,
\end{equation}
which is a symmetric $2$-tensor on $\mcl{B}$.

Now we recall the flow \eqref{eq:hrf}: if $\phi : (\mcl{B},g) \rightarrow (\mcl{N},\gamma)$ is a map of Riemannian manifolds, the \textit{harmonic-Ricci flow} is the coupled system
\begin{align*}
\partial_t g    &= -2 \Rc + 2 c \, \nabla \phi \otimes \nabla \phi \\
\partial_t \phi &= \tau_{g,\gamma} \phi 
\end{align*}
where $c=c(t) \geq 0$ is a coupling function.  This flow is also sometimes called the $(RH)_\alpha$ flow (when $c = \alpha)$.  We will assume that $c(t)$ is non-increasing.  As mentioned above, this flow was introduced in \cite{Muller2012} and is a generalization of one studied in \cite{List2008}.  

Here are some examples of the flow.  In studying expanding Ricci solitons on homogeneous spaces, Lott considered as a model a special type of vector bundle \cite{Lott2007}.  Let $\mcl{M}$ be an $\R^N$-vector bundle with flat connection, flat metric $G$ on the fibers, and Riemannian base $(\mcl{B}^n,g)$.  Assume that the connection preserves fiberwise volume forms.  Lott showed that the soliton equation becomes a pair of equations.  One is a soliton-like equation for $g$.  The other is an equation for $G$, which says that $G$ is harmonic when interpreted as a map $G : \mcl{B} \rightarrow \big(\SL(N,\R)/\SO(N),\theta \big)$.  Here, $\theta$ is the natural metric induced by $G$; see \eqref{eq:S-metric}.

In fact, more is true.  Ricci flow on such bundles is the coupled flow
\begin{align*}
\partial_t g &= -2 \Rc + \frac{1}{4} \nabla G \otimes \nabla G \\
\partial_t G &=  \tau_{g,\theta} G
\end{align*}
which is the harmonic-Ricci flow on $(\mcl{B},g)$ with map $G$ and $c=1/8$; see \cite{Williams2013-hrf}.  This is a special case of the coupled system considered in Section \ref{sec:irf}.  All 3D and 4D homogeneous spaces admitting expanding Ricci solitons have this bundle structure, so the corresponding Ricci flow solutions are harmonic-Ricci flow solutions.

Ricci flow on a warped product $(\mcl{B}^n \times \mcl{S}^1,g + e^{2u} d\theta)$, after modification by diffeomorphisms, has the form:
\[ \begin{aligned}    
\partial_t g &= -2 \Rc + 2 \, du \otimes du \\
\partial_t u &= \Delta u
\end{aligned} \]
This is harmonic-Ricci flow with target $\R$ and $c=1$, and was studied by Lott and \v{S}e\v{s}um \cite{LottSesum2011}.  It is also a special case of the system considered in Section \ref{sec:wp}.

\subsection{Stability}  

In this section we follow the outline given in the introduction to prove Theorem \ref{thm:hrf-stab}.

We transform the system into one whose fixed points include pairs $(g_0,\phi_0)$ with $g_0$ Einstein and $\phi_0$ constant.  Suppose that $\big(\overline{g}(\overline{t}),\overline{\phi}(\overline{t})\big)$ is a solution of \eqref{eq:hrf} for $\overline{t} \geq 0$.  For $\lambda < 0$, let $\sigma(\overline{t}) := 1-2\lambda \overline{t}$ and consider
\[ \big( \overline{g}(\overline{t}),\overline{\phi}(\overline{t}) \big) \longmapsto \big( g(t), \phi(t) \big), \]
where
\[ 
t := -\frac{1}{2\lambda} \log \sigma(\overline{t}), \quad
g := \sigma^{-1} \overline{g}, \quad 
\phi := \overline{\phi}.
 \]
A straightforward calculation in the manner of \cite{ChowKnopf2004}*{Section 9.1} shows that this transformation results in the modified flow
\begin{equation}\label{eq:cnhrf}
\begin{aligned}
\partial_t g    &= -2 \Rc + 2 c \, \nabla \phi \otimes \nabla \phi + 2\lambda g \\
\partial_t \phi &= \tau_{g,\gamma} \phi
\end{aligned} 
\end{equation}
Call this system the \textit{curvature-normalized harmonic-Ricci flow}.

We next use the DeTurck trick to make the system \eqref{eq:cnhrf} strictly parabolic.  That is, we pull back by diffeomorphisms generated by carefully chosen vector fields, which has the effect of subtracting a Lie derivative term from both equations in \eqref{eq:cnhrf}.  To this end, fix a background metric $g_0$ (which we can take to be a fixed-point metric; see below) and define a vector field $W$ depending on $g_0$ and $g(t)$ by
\[ W^k := g^{ij}(\Gamma_{ij}^k - {}^{g_0}\Gamma_{ij}^k) \]
for $k = 1,\dots,n$.  Let $F_{t}$ be diffeomorphisms generated by $W(t)$, with initial condition $F_0 = \id$.  The one-parameter family $\big(F_t^* g(t),F_t^* \phi(t)\big)$ is the solution of  
\begin{equation}\label{eq:cnhrdtf}
\begin{aligned}
\partial_t g    &= -2 \Rc + 2c \, \nabla \phi \otimes \nabla \phi +2\lambda g - \mcl{L}_W g,  \\
\partial_t \phi &= \tau_{g,h} \phi - \mcl{L}_W \phi,  
\end{aligned} 
\end{equation}
Call this system the \textit{curvature-normalized harmonic-Ricci-DeTurck flow}.

A stationary solution of \eqref{eq:cnhrf} is also a stationary solution of \eqref{eq:cnhrdtf}, and we can describe a large class of such fixed points.

\begin{lem}\label{lem:hrf-fixed-points}
Suppose that $(\mcl{B}^n,g_0)$ is closed and $\lambda$-Einstein, and let $\phi_0 : (\mcl{B},g) \rightarrow (\mcl{N},\gamma)$ be a constant map.  The pair $(g_0,\phi_0)$ is a stationary solution of the flows \eqref{eq:cnhrf} and \eqref{eq:cnhrdtf}.
\end{lem}

To analyze the stability of such a fixed point $(g_0,\phi_0)$ of this flow, we must compute the linearization.  Let
$\big( \widetilde{g}(\epsilon),\widetilde{\phi}(\epsilon)\big)$ be a variation of $(g_0,\phi_0)$ such that
\begin{equation}\label{eq:hrf-var}
\begin{aligned}
\widetilde{g}(0) &= g_0, 
                 &\partial_\epsilon \big|_{\epsilon=0} \widetilde{g}(\epsilon) 
								 &= h \in \mcl{S}^2 \mcl{M}, \\
\widetilde{\phi}(0) &= \phi_0, 
                    &\partial_\epsilon \big|_{\epsilon=0} \widetilde{\phi}(\epsilon) 
										&= \psi \in C^\infty(\phi_0^* T\mcl{N}).
\end{aligned}
\end{equation}
More explicitly, $\widetilde{\phi}(x,\epsilon) = \exp_{\phi_0} (\epsilon \psi(x))$.  Let $\Delta_{\ell}$ denote the Lichnerowicz Laplacian acting on symmetric
$(2,0)$-tensor fields.  Its components are
\[ \Delta_{\ell}h_{ij}=\Delta h_{ij}+2R_{ipqj}h^{pq}-R_{i}^{k}h_{kj}-R_{j}^{k}h_{ik}. \]

\begin{lem}\label{lem:hrf-deturck-lin}
The linearization of the curvature-normalized harmonic-Ricci DeTurck flow at a fixed point $(g_0,\phi_0)$, where $g_0$ is $\lambda$-Einstein and $\phi_0$ is constant, is the autonomous, self-adjoint, strictly parabolic system
\begin{subequations}
\begin{align}
\partial_t    h &= \mbf{L}_{0} h    := \Delta_{\ell} h + 2\lambda h \label{eq:hrf-L0} \\
\partial_t \psi &= \mbf{L}_{1} \psi := \Delta \psi \label{eq:hrf-L1}
\end{align}
\end{subequations}
where $\mbf{L}_1 = \Delta$ satisfies $(\Delta \psi)^\alpha = \Delta (\psi^\alpha)$ in local coordinates.
\end{lem}

\begin{proof}
With a variation as in \eqref{eq:hrf-var}, we must compute
\[ \partial_\epsilon \big|_{\epsilon = 0} \Big( \partial_t \widetilde{g}(\epsilon) \Big),
\qquad \partial_\epsilon \big|_{\epsilon = 0} \Big( \partial_t \widetilde{\phi}(\epsilon) \Big). \]
Such computations involve standard variational formulas for geometric objects like $g^{-1}$, $\Gamma$, $\Rc$, and $R$.  See \cite{ChowKnopf2004}*{Section 3.1}, for example.  The first equation is similar to one considered in \cite{Knopf2009}*{Lemmas 3 and 4}; in particular, the Lie derivative from the DeTurck trick precisely cancels the unpleasant terms from the linearization of $\Rc$.  In local coordinates, we use \eqref{eq:phi-quad} to see that the linearization of the term involving $\phi$ vanishes, since
 \begin{align*}
\partial_\epsilon \big|_{\epsilon = 0} \Big( \nabla \widetilde{\phi}(\epsilon) \otimes \nabla \widetilde{\phi}(\epsilon) \Big)_{ij} 
&= \partial_\epsilon \big|_{\epsilon = 0} \Big(
\gamma_{\alpha\beta} \partial_i \widetilde{\phi}(\epsilon)^\alpha \partial_j \widetilde{\phi}(\epsilon)^\beta \Big) \\
&= \gamma_{\alpha\beta} \Big( \partial_i \psi^\alpha \partial_j \phi_0^\beta + \partial_i \phi_0^\alpha \partial_j \psi^\beta \Big) = 0.
\end{align*}
For the second equation, we use the coordinate expression for the tension field from \eqref{eq:tension} to see that
\begin{align*}
\partial_\epsilon \big|_{\epsilon = 0} \Big( \partial_t \widetilde{\phi}(\epsilon) \Big)^\alpha
&= \partial_\epsilon \big|_{\epsilon = 0} \tau_{\widetilde{g}(\epsilon),\gamma}
\widetilde{\phi}(\epsilon)^\alpha \\
&= g^{ij} ( \partial_i \partial_j \psi^\alpha - {}^g \Gamma_{ij}^k \partial_k \psi^\alpha)
= \Delta \psi^\alpha. 
\end{align*}
Also,
\[ \mcl{L}_W \phi = \nabla \phi(W) = \partial_i \phi^\alpha W_i \, \partial_\alpha |_\phi, \]
and since $W(0)=0$, we have
\[ \partial_\epsilon \big|_{\epsilon=0} \partial_i \widetilde{\phi}^\alpha \widetilde{W}_i = 0, \]
giving \eqref{eq:hrf-L1}.
\end{proof}


We now wish to consider linear stability of a fixed point $(g_0,\phi_0)$ of the curvature-normalized harmonic-Ricci-DeTurck flow with respect to the system (\ref{eq:hrf-L0}-\ref{eq:hrf-L1}).  We will assume that $\phi_0$ is constant and $g_0$ is a $\lambda$-Einstein metric that is linearly stable with respect to \eqref{eq:L0}.  First, however, let us consider an example of when this last property is satisfied.

\begin{lem}\label{lem:linear-stab}

Suppose that $g$ is $\lambda$-Einstein, and that there exists $K < 0$ such that $\sec \leq K$.  Then for $\mbf{L}_0 h = \Delta_\ell h + 2\lambda h$, we have $(\mbf{L}_0 h,h) \leq K(n-2) \|h\|^2 < 0$.
\end{lem}

\begin{proof}
First, write a symmetric $2$-tensor $h$ as $h = g(\widetilde{h} \cdot,\cdot)$ and let $\{e_i\}$ be a local orthonormal basis of eigenvectors for $\widetilde{h}$.  That is, $\widetilde{h}(e_i) = \mu_i e_i$.  Now, we write part of $\langle \mbf{L}_0 h,h \rangle$ in components with respect this basis.  Using that $g$ is $\lambda$-Einstein, we have
\begin{align*}
&\sum_{i,j,k,\ell} R_{ijk\ell} h^{i\ell} h^{jk} - \sum_{i,j,k} R_i^j h_j^k h_k^i + 2\lambda \sum_{i,j} h_{ij} h^{ij} \\
&= \sum_{i,j} R_{ijji} \mu_i \mu_j + \lambda \sum_{i,j} h_{ij} h^{ij} \\
&= \sum_{i,j} \sec(e_i,e_j) \mu_i \mu_j + \sum_{i,j} \sec(e_i,e_j) \mu_i^2 \\
&= \half \sum_{i,j} \sec(e_i,e_j) 2 \mu_i \mu_j + \half \sum_{i,j} \sec(e_i,e_j) \mu_i^2 + \half \sum_{i,j} \sec(e_i,e_j) \mu_j^2 \\
&= \half \sum_{i,j} \sec(e_i,e_j) (\mu_i + \mu_j)^2.
\end{align*}
Now, integrating by parts and using Koiso's Bochner formula \cite{Koiso1978} together with the above equation,
\begin{align*}
(\mbf{L}_0 h,h)
&= -\|\nabla h\|^2 + \int R_{ijk\ell} h^{ij} h^{kl} - \int R_i^j h_j^k h_k^i 
   + \half \int \sum_{i,j} \sec(e_i,e_j) (\mu_i + \mu_j)^2 \\
&= -\half \|T\|^2 - \|\delta h\|^2 + K(n-2) \|h\|^2 + K \|\tr h\|^2 \\
&\leq K(n-2) \|h\|^2,
\end{align*}
as desired.
\end{proof}

Regarding linear stability of $\phi_0$, we have the following, which follows from integration by parts.

\begin{lem}\label{lem:hrf-stable}
The constant map $\phi_0$ is weakly linearly stable with respect to the operator $\mathbf{L}_{1}$.  Its null eigenspace is the space of constant variations $\psi \in C^\infty(\phi_0^*T\mcl{N})$, whose dimension is equal to $\dim \mcl{N}$.
\end{lem}


We now turn to the proof of Theorem \ref{thm:hrf-stab}.  See the appendix for the statement of Simonett's theorem, \cite{Knopf2009}*{Section 2} for a more detailed description of its application as used here, and \cite{Simonett1995} for the original statement.  

If $\mcl{V} \rightarrow \mcl{B}$ is a vector bundle, let $\mfr{h}^{r+\rho}(\mcl{V})$ denote the completion of the vector space $C^\infty(\mcl{V})$ with respect to the $r+\rho$ little-H\"older norm.  For fixed $0 < \sigma < \rho < 1$, consider the following densely and continuously embedded spaces:
\begin{align*}
\mbb{E}_0 &:=\mfr{h}^{0+\sigma}(S^2 \mcl{B}) \times 
            \mfr{h}^{0+\sigma}(\phi_0^* T\mcl{N}) \\
\cup \, & \\            
\mbb{X}_0 &:=\mfr{h}^{0+\rho}(S^2 \mcl{B}) \times 
            \mfr{h}^{0+\rho}(\phi_0^* T\mcl{N}) \\
\cup \, & \\               
\mbb{E}_1 &:=\mfr{h}^{2+\sigma}(S^2 \mcl{B}) \times 
            \mfr{h}^{2+\sigma}(\phi_0^* T\mcl{N}) \\
\cup \, & \\               
\mbb{X}_1 &:=\mfr{h}^{2+\rho}(S^2 \mcl{B}) \times 
            \mfr{h}^{2+\rho}(\phi_0^* T\mcl{N})
\end{align*}
For fixed $1/2\leq\beta<\alpha<1$, define the continuous interpolation spaces
\begin{align*}
\mbb{X}_\beta := (\mbb{X}_0, \mbb{X}_1)_\beta,\quad
\mbb{X}_\alpha := (\mbb{X}_0, \mbb{X}_1)_\alpha.
\end{align*}
For fixed $0<\epsilon\ll 1$, let $\mbb{G}_\beta$ be the open $\epsilon$-ball around $(g_0,\phi_0)$ in $\mbb{X}_\beta$, and define $\mbb{G}_\alpha := \mbb{G}_\beta \cap \mbb{X}_\alpha$.

\begin{proof}[Proof of Theorem \ref{thm:hrf-stab}.]
This proof follows that of Theorem 1 in \cite{Knopf2009}, and due to the abstract nature of Simonett's Theorem, essentially all of the details go through unchanged.  Indeed, because we assume that the metric is strictly linearly stable, it does not correspond to any center manifolds.  Still, we briefly describe the four main steps in the proof.

First, one must show that the complexification of the operator $\mbf{L} := (\mbf{L}_0,\mbf{L}_1,\mbf{L}_2)$ in \eqref{eq:hrf-L0}-\eqref{eq:hrf-L1} (which acts component-wise) is sectorial.  This holds exactly as in  \cite{Knopf2009}.  With this established, one checks that conditions (\ref{SimFirst})-(\ref{SimLast}) of Simonett's theorem hold, and this follows exactly as in \cite{Knopf2009}*{Lemmas 1 and 2}.  The second step is then to apply Simonett's theorem (Theorem \ref{thm:simonett}) to obtain, in particular, existence of exponentially attractive local center manifolds.  

Finally, one proves the uniqueness of a smooth center manifold consisting of fixed points of the flow \eqref{eq:cnhrdtf}, and convergence of the solution to a point in that center manifold.  Again, this is actually simpler in our case because of the strict linear stability of the metrics, and the arguments of \cite{Knopf2009} carry through without modification.
\end{proof}


\section{Ricci flow on warped products}\label{sec:wp}

\subsection{Setup}

Consider a multiply-warped product $(\mcl{M},\mbf{g})$ as described above, where
\[ \mcl{M} := \mcl{B} \times \prod_{\alpha=1}^m \mcl{F}_\alpha, \qquad
   \mbf{g} := g + \sum_{\alpha=1}^m e^{2\phi_\alpha} g_\alpha. \]
Let $\mcl{H},\mcl{V}_\alpha \subset T\mcl{M}$ denote the horizontal and $\alpha$-vertical distributions, respectively.  The Ricci curvature of $\mbf{g}$ is
\begin{equation}\label{eq:mwp-ric}
\begin{aligned}
{}^\mbf{g}\! \Rc|_{\mcl{H} \times \mcl{H}}
&= {}^g\! \Rc - \sum_{\alpha=1}^m n_\alpha \big[ \Hess(\phi_\alpha) + d\phi_\alpha \otimes d\phi_\alpha \big] \\
{}^\mbf{g}\! \Rc|_{\mcl{V_\alpha} \times \mcl{V_\alpha}}
&= {}^{g_\alpha}\! \Rc - \left[ \Delta \phi_\alpha + \sum_{\beta=1}^m n_\beta \langle d\phi_\alpha,d\phi_\beta \rangle \right] e^{2\phi_\alpha} g_\alpha \\
{}^\mbf{g}\! \Rc|_{\mcl{H} \times \mcl{V}_\alpha}
&= 0 \qquad \text{for all } \alpha =1,\dots,m \\
{}^\mbf{g}\! \Rc|_{\mcl{V}_\alpha \times \mcl{V}_\beta}
&= 0 \qquad \text{for all } \alpha,\beta =1,\dots,m \text{ with } \alpha \neq \beta 
\end{aligned}
\end{equation}
See, e.g., \cite{DobarroUnal2005}.  We wish to describe Ricci flow on $(\mcl{M},\mbf{g})$ in terms of the evolution of $g$ and the warping functions $\phi_\alpha$.

\begin{prop} Let $(\mcl{M} = \mcl{B} \times \prod_\alpha \mcl{F}_\alpha, \mbf{g}(t))$ be a solution to Ricci flow, with $\mcl{B}$ closed and $\mbf{g}(0) = g(0)+ \sum_\alpha e^{2\phi_\alpha(0)} g_\alpha(0)$ a multiply-warped product.  If $(\mcl{F}_\alpha,g_\alpha)$ is closed and $\mu_\alpha$-Einstein for each $\alpha$, then each $g_\alpha = g_\alpha(0)$ is constant under the flow, $\mbf{g}(t)$ is a multiply-warped product, and the evolutions of $g$ and the  $\phi_\alpha$ are given by
\begin{equation}\label{eq:mwprf-pre}
\begin{aligned}
\partial_t g    &= -2 \, {}^g\! \Rc + 2 \sum_{\alpha=1}^m n_\alpha \big[ \Hess(\phi_\alpha) +  d\phi_\alpha \otimes d\phi_\alpha \big] \\
\partial_t \phi_\alpha &= \Delta \phi_\alpha + \sum_{\beta=1}^m n_\alpha \langle d\phi_\alpha,d\phi_\beta \rangle - \mu_\alpha e^{-2\phi_\alpha} \qquad \alpha=1,\dots,m
\end{aligned}
\end{equation}
\end{prop}

\begin{proof}
Define $\widetilde{\mbf{g}}(t) = g(t) + \sum_\alpha e^{2\phi_\alpha(t)} g_\alpha$, where $g(t)$ and $\phi_\alpha(t)$ are solutions of \eqref{eq:mwprf-pre} with $g(0) = g_0$ and $\phi_\alpha(0) = (\phi_\alpha)_0$.  Using \eqref{eq:mwp-ric}, we see that the evolution of $\widetilde{\mbf{g}}(t)$ is
\begin{align*}
\partial_t \widetilde{\mbf{g}}(t)
&= \partial_t g(t) + \sum_\alpha \partial_t (e^{2\phi_\alpha(t)} g_\alpha) \\
&= -2 \, {}^g\! \Rc + 2 \sum_\alpha n_\alpha \big[ \Hess(\phi_\alpha) + d \phi_\alpha \otimes d\phi_\alpha \big] 
  + 2 \sum_\alpha e^{2\phi_\alpha(t)} \partial_t \phi_\alpha(t) \, g_\alpha \\
&= -2 \, {}^{\widetilde{\mbf{g}}}\! \Rc|_{\mcl{H} \times \mcl{H}} + 2 \sum_\alpha \left[ \Delta \phi_\alpha + \sum_\beta n_\beta \langle d\phi_\alpha,d\phi_\beta \rangle  - \mu_\alpha e^{-2\phi_\alpha(t)} \right] e^{2\phi_\alpha(t)} g_\alpha \\
&= -2 \, {}^{\widetilde{\mbf{g}}}\! \Rc|_{\mcl{H} \times \mcl{H}} + 2 \sum_\alpha \Big[ \Delta \phi_\alpha + \sum_\beta n_\beta \langle d\phi_\alpha,d\phi_\beta \rangle \Big] e^{2\phi(t)} g_\alpha - 2 \sum_\alpha \mu_\alpha g_\alpha \\
&= -2 \, {}^{\widetilde{\mbf{g}}}\! \Rc|_{\mcl{H} \times \mcl{H}} - 2 \sum_\alpha {}^{\widetilde{\mbf{g}}} \! \Rc|_{\mcl{V_\alpha} \times \mcl{V_\alpha}} \\
&= -2 \, {}^{\widetilde{\mbf{g}}}\! \Rc,
\end{align*}
since $\mu_\alpha g_\alpha = {}^{g_\alpha} \!\Rc$.  This means $\widetilde{\mbf{g}}(t)$ solves Ricci flow with $\widetilde{\mbf{g}}(0) = g_0 + \sum_\alpha e^{2(\phi_\alpha)_0} g_\alpha$.  By uniqueness of solutions of Ricci flow, for any solution $\mbf{g}(t)$ of Ricci flow with $\mbf{g}(0) = g_0 + \sum_\alpha e^{2(\phi_\alpha)_0} g_\alpha$, we must have $\widetilde{\mbf{g}}(t) = \mbf{g}(t)$.  This means the multiply-warped product structure is preserved and $\mbf{g}(t) = g(t) + \sum_\alpha e^{2\phi_\alpha(t)} g_\alpha$ must satisfy \eqref{eq:mwprf-pre}.
\end{proof}

We can simplify the system \eqref{eq:mwprf-pre}.  For $f = -\sum_\alpha n_\alpha \phi_\alpha \in C^\infty(\mcl{M})$, the following Lie derivatives are easily computed:
\[ \begin{aligned}
\mcl{L}_{\nabla f} g    &= -2 \sum_\alpha \Hess(\phi_\alpha), \\
\mcl{L}_{\nabla f} \phi_\alpha &= - \sum_\beta n_\beta \langle d\phi_\alpha,d\phi_\beta \rangle. 
\end{aligned} \]
Pulling back by diffeomorphisms generated by $\nabla f$ amounts to adding these Lie derivatives to the equations in \eqref{eq:mwprf-pre}.  From this we obtain the system
\begin{equation}\label{eq:mwprf}
\begin{aligned}
\partial_t g      &= -2 \, {}^g\! \Rc + 2 \sum_\alpha n_\alpha \, d\phi_\alpha \otimes d\phi_\alpha \\
\partial_t \phi_\alpha &= \Delta \phi_\alpha - \mu_\alpha e^{-2\phi_\alpha} \qquad \alpha=1,\dots,m
\end{aligned}
\end{equation}
which we call \textit{multiply-warped product Ricci flow}.  If we define a map 
\[ \Phi : \mcl{B} \longrightarrow \R^m \]
\[ \Phi(x) := \big(\phi_1(x),\dots,\phi_m(x)\big) \]
then we can abbreviate a solution of \eqref{eq:mwprf} as $\big(g(t),\Phi(t)\big)$.

\begin{rem} 
The system \eqref{eq:mwprf} does not depend on the fibers $(\mcl{F}_\alpha,g_\alpha)$ at all, except for the Einstein constants $\mu_\alpha$ and the dimensions $n_\alpha$.  Considering that system abstractly (that is, outside the context of Ricci flow on $\mcl{M}$), we can therefore allow the fibers $\mcl{F}_\alpha$ to be non-compact.  For example, a modification of the DeTurck trick shows short-time existence of solutions in that case.
\end{rem}

\subsection{Estimates} 

We need to understand the evolution of various geometric quantities under the flow \eqref{eq:wprf}.  The following equations can be proved in a manner similar to those for harmonic-Ricci flow with 1-dimensional target, as found in \cite{List2008}.  

\begin{lem}\label{lem:hrf-evo}
Let $\big(g(t),\Phi(t)\big)$ be a solution of \eqref{eq:wprf}.  We have the following evolution equations, for $\alpha = 1,\dots,m$.
\[ \begin{aligned}
\partial_t \partial_i \phi_\alpha
&= \Delta \partial_i \phi_\alpha 
- R_i^p \partial_p \phi_\alpha 
+ 2 \mu_\alpha e^{-2\phi_\alpha} \partial_i \phi_\alpha \\
\partial_t |d\phi_\alpha|^2
&= \Delta |d\phi_\alpha|^2 
- 2 |\Hess \phi_\alpha|^2 
- 2\sum_\beta n_\beta \langle d\phi_\alpha,d\phi_\beta \rangle^2 
+ 4 \mu_\alpha e^{-2\phi_\alpha} |d\phi_\alpha|^2 
\end{aligned} \]
\end{lem}


We will use the estimates from the lemma with the following version of the Maximum Principle.

\begin{thm}
Suppose $g(t)$ is a family of metrics on a closed manifold $\mcl{M}^n$, $X(t)$ is a time-dependent vector field on $\mcl{M}$, and $F : \R \times [0,T) \rightarrow \R$ is a Lipschitz continuous function.  Consider the semi-linear heat equation
\begin{equation}\label{eq:heat}
\partial_t u = \Delta_{g(t)} u + \langle X(t), \nabla u \rangle + F(u,t),
\end{equation}
and the corresponding ordinary differential equation
\begin{equation}\label{eq:ode}
\frac{d}{dt} U = F(U,t),
\end{equation}
for functions $u : \mcl{M} \times [0,T) \rightarrow \R$, $U : [0,T') \rightarrow \R$.

Let $u(x,t)$ be a $C^2$ solution of \eqref{eq:heat}, and let $U_1$ and $U_2$ solve \eqref{eq:ode} with $U_1(0) = \min_\mcl{M} u(x,0)$ and $U_2(0)= \max_\mcl{M} u(x,0)$, respectively.  In particular,
\[ U_1(0) \leq u(x,0) \leq U_2(0) \]
for all $x \in \mcl{M}$.  Then as long as these functions exist, 
\[ U_1(t) \leq u(x,t) \leq U_2(t), \]
for all $x \in \mcl{M}$.
\end{thm}

\begin{lem}\label{lem:phi-estimate}
Suppose that $\big(g(t),\Phi(t)\big)$ is a solution of \eqref{eq:mwprf} on $\mcl{B} \times \prod_\alpha \mcl{F}_\alpha$, where $\mcl{B}$ is closed and each $(\mcl{F}_\alpha,g_\alpha)$ is Einstein with constant $\mu_\alpha < 0$.  Then for each $\alpha$ there are constants $c_\alpha,d_\alpha,C_\alpha$ such that
\begin{align}
e^{2c_\alpha} - 2 \mu_\alpha t \leq\, e^{2\phi_\alpha(x,t)} \leq e^{2d_\alpha} - 2 \mu_\alpha t \label{fn-bounds-neg} \\
0 \leq |d\phi_\alpha(x,t)|^2 \leq \frac{C_\alpha}{(t+1)^2} \label{eq:d-psi-decay}
\end{align}
for all for all $x \in \mcl{B}$ and $t$ where the solution exists.
\end{lem}

\begin{proof}
The \textsc{ode} associated with the evolution of $\phi_\alpha$ is
\[ \frac{d}{dt} U = F_\alpha(U) := -\mu_\alpha e^{-2U}. \]
For initial data $U(0) = c$, this has solution $U(t) = \half\log( e^{2c} - 2\mu_\alpha t)$.  Let $U_1$ and $U_2$ solve the \textsc{ode} with initial data $c_\alpha := \min_\mcl{B} \phi_\alpha(x,0)$ and $d_\alpha := \max_\mcl{B} \phi_\alpha(x,0)$, respectively.  Then
\begin{equation}\label{eq:fn-bounds-log}
\frac{1}{2} \log ( e^{2c_\alpha} - 2\mu_\alpha t) = U_1(t) \leq \phi_\alpha(x,t) \leq U_2(t) = \frac{1}{2} \log ( e^{2d_\alpha} - 2\mu_\alpha t)
\end{equation}
for as long as the functions exist, and for all $x \in \mcl{B}$.  Exponentiating, we see that
\[ e^{2c_\alpha} - 2 \mu_\alpha t \leq e^{2\phi_\alpha(x,t)} \leq e^{2d_\alpha} - 2 \mu_\alpha t. \]
When $\mu_\alpha < 0$, this means the warping factor $e^{2\phi_\alpha}$ grows to infinity.  Also note that when $\mu_\alpha = 0$ or $\mu_\alpha > 0$, $e^{2\phi_\alpha(x,t)}$ remains bounded or reaches zero in finite time, respectively.

Now let us find bounds on $|d\phi_\alpha|^2$.  First, from the behavior of $\phi_\alpha$ above, we have
\[ 4 \mu_\alpha e^{-2\phi_\alpha} \leq \frac{4 \mu_\alpha}{e^{2d_\alpha}-2\mu_\alpha t}, \]
Then $|d\phi_\alpha|^2$ evolves according to
\begin{align*}
\partial_t |d\phi_\alpha|^2
&= \Delta |d\phi_\alpha|^2 
- 2|\Hess \phi_\alpha|^2 
- 2\sum_\beta n_\alpha \langle d\phi_\alpha,d\phi_\beta \rangle^2 
+ 4 \mu_\alpha e^{-2\phi_\alpha} |d\phi_\alpha|^2 \\ 
&\leq \Delta |d\phi_\alpha|^2 + \frac{4 \mu_\alpha}{e^{2d_\alpha} - 2\mu_\alpha t} |d\phi_\alpha|^2.
\end{align*}
That is, $v := |d\phi_\alpha|^2$ is a subsolution of $\partial_t u = \Delta u + G_\alpha(u,t)$, where 
\[ G_\alpha(u,t) := \frac{4 \mu_\alpha}{e^{2d_\alpha} - 2\mu_\alpha t} u. \]
The corresponding \textsc{ode}, with initial data $U(0)=U_0$, has solution
\[ U(t) = \frac{e^{4d_\alpha} U_0}{(e^{2d_\alpha}-2\mu_\alpha t)^2}. \]
Taking $U_0 := \max_\mcl{B} |d\phi_\alpha(x,0)|^2$, the Maximum Principle gives
\begin{equation} 
0 \leq |d\phi_\alpha(x,t)|^2 \leq \frac{e^{4d_\alpha} U_0}{(e^{2d_\alpha}-2\mu_\alpha t)^2} \leq \frac{C_\alpha}{(t+1)^2} 
\end{equation}
for all time and all $x \in \mcl{B}$, for some $C_\alpha>0$.  
\end{proof}

\subsection{Stability}

In this section we follow the outline given in the introduction to prove Theorem \ref{thm:wprf-stab}.

When $\mu_\alpha<0$, the warping functions $\phi_\alpha$ grow in a controlled way:~the lower and upper bounds of $\phi_\alpha$ on $\mcl{B}$ both go to infinity by \eqref{fn-bounds-neg}, but we also have $|d\phi_\alpha|^2 \rightarrow 0$ as $t \rightarrow \infty$ by \eqref{eq:d-psi-decay}.  It is natural, therefore, to hope that each $\phi_\alpha$ converges to a constant function, but the growth condition implies that this constant should be $\infty$.  This means that some kind of normalization is needed for $\phi_\alpha$ (as well as for $g$).  

Therefore, let $\big(\overline{g}(\overline{t}),\overline{\Phi}(\overline{t})\big)$ solve \eqref{eq:mwprf} for $\overline{t} \geq 0$.  For $\lambda < 0$, let $\sigma(\overline{t}) := 1-2\lambda \overline{t}$, let $A_\alpha$ be constants, and consider the transformation  
\[ \big( \overline{g}(\overline{t}), \overline{\Phi}(\overline{t}) \big) \mapsto \big( g(t), \Phi(t) \big), \]
where
\[ 
t := -\frac{1}{2\lambda} \log \sigma(\overline{t}), \quad
g := \sigma^{-1} \overline{g}, \quad 
\phi_\alpha := 
\begin{cases}
\overline{\phi}_\alpha - \frac{1}{2} \log \sigma + A_\alpha & \text{if } \mu_\alpha < 0 \\
\overline{\phi}_\alpha & \text{if } \mu_\alpha = 0
\end{cases}.
\]
Note that $\phi_\alpha$ is simply a translate of $\overline{\phi}_\alpha$.  As such, $d\phi_\alpha = d\overline{\phi}_\alpha$, and so all spatial derivative behavior of $\phi_\alpha$ is the same as that of $\overline{\phi}_\alpha$.  Additionally, when $\mu_\alpha < 0$ it is easy to see from \eqref{fn-bounds-neg} that
\[ \lim_{t \rightarrow \infty} \phi_\alpha(x,t) = A_\alpha \]
for all $x \in \mcl{B}$.

A computation shows that this transformation results in the system
\begin{equation}\label{eq:cnmwprf}
\begin{aligned}
\partial_t g    &= -2 \Rc + 2 \sum_\alpha n_\alpha d\phi_\alpha \otimes d\phi_\alpha + 2\lambda g  \\
\partial_t \phi_\alpha &= 
\begin{cases}
\Delta \phi_\alpha 
- \mu_\alpha e^{-2(\phi_\alpha-A_\alpha)} 
+ \lambda & \text{if } \mu_\alpha < 0 \\
\Delta \phi_\alpha & \text{if } \mu_\alpha = 0
\end{cases}, 
\qquad \alpha = 1,\dots,m
\end{aligned}
\end{equation}
which we call \textit{curvature-normalized multiply-warped product Ricci flow}.

As before, we use the DeTurck trick to make the \eqref{eq:cnmwprf} system strictly parabolic.  Fix a background metric $g_0$ (which we can take to be a fixed-point metric) and define a vector field $W$ depending on $g(t)$ by
\[ W^k := g^{ij}(\Gamma_{ij}^k - {}^{g_0}\Gamma_{ij}^k) \]
for $k = 1,\dots,n$.  Let $F_{t}$ be diffeomorphisms generated by $W(t)$, with initial condition $F_0 = \id$.  The one-parameter family $\big(F_t^* g(t), F_t^* \Phi(t)\big)$ is the solution of 
\begin{equation}\label{eq:cnmwprdtf}
\begin{aligned}
\partial_t g    &= -2 \Rc + 2 \sum_\alpha n_\alpha d\phi_\alpha \otimes d\phi_\alpha + 2\lambda g - \mcl{L}_W g \\
\partial_t \phi_\alpha &= 
\begin{cases}
\Delta \phi_\alpha 
- \mu_\alpha e^{-2(\phi_\alpha-A_\alpha)} 
+ \lambda - \mcl{L}_W \phi_\alpha & \text{if } \mu_\alpha < 0 \\
\Delta \phi_\alpha - \mcl{L}_W \phi_\alpha & \text{if } \mu_\alpha = 0
\end{cases}, 
\qquad \alpha = 1,\dots,m
\end{aligned}
\end{equation}
which we call the \textit{curvature-normalized multiply-warped product Ricci--DeTurck flow}.  

A stationary solution of (\ref{eq:cnmwprf}) is also a stationary solution of the curvature-normalized multiply-warped product Ricci-DeTurck flow, and we can describe a large class of such fixed points.

\begin{lem}\label{lem:wprf-fixed-points}
Suppose that $(\mcl{B}^n \times \prod \mcl{F}_\alpha,g_0 + e^{2(\phi_\alpha)_0} g_\alpha)$ is a multiply-warped product with $(\mcl{B},g)$ closed and $\lambda$-Einstein and suppose that each $\big(\mcl{F}_\alpha,(g_\alpha)_0\big)$ is $\mu_\alpha$-Einstein with either $\mu_\alpha=\lambda$ or $\mu_\alpha=0$.  Let each $(\phi_\alpha)_0 : \mcl{M} \rightarrow \R$ be a constant function, and let $A_\alpha = (\phi_\alpha)_0$.  Then $(g_0,\Phi_0)$ a stationary solution of \eqref{eq:cnmwprf} and \eqref{eq:cnmwprdtf}.
\end{lem}


To analyze the stability of a fixed point$(g_0,\Phi_0)$ of flow \eqref{eq:cnmwprdtf}, we must compute the linearization of the flow.  Let $\big(\widetilde{g}(\epsilon),\widetilde{\Phi}(\epsilon)\big)$ be a variation of $(g_0,\Phi_0)$ such that
\begin{equation}\label{eq:wp-var}
\begin{aligned}
\widetilde{g}(0) &= g_0, 
                 &\partial_\epsilon \big|_{\epsilon=0} \widetilde{g}(\epsilon) 
								 &= h \in \mcl{S}^2 \mcl{B}, \\
\widetilde{\phi}_\alpha(0) &= (\phi_\alpha)_0, 
											     &\partial_\epsilon \big|_{\epsilon=0} \widetilde{\phi}_\alpha(\epsilon) 
													 &= \psi_\alpha \in C^\infty(\mcl{B})
\end{aligned}
\end{equation}
for $\alpha = 1,\dots,m$.

\begin{lem}\label{lem:cnmwprdtf-lin}
The linearization of the curvature-normalized multiply-warped product Ricci-DeTurck flow at a fixed point $(g_0,\Phi_0)$ where $g_0$ is $\lambda$-Einstein and $\Phi_0$ constant is the autonomous, self-adjoint, strictly parabolic system
\begin{subequations}
\begin{align}
\partial_t    h & = \mathbf{L}_{0} h := \Delta_{\ell} h + 2\lambda h \label{eq:wp-lin-h} \\
\partial_t \psi_\alpha & = \mathbf{L}_\alpha \psi_\alpha := 
\begin{cases}                                                        
\Delta \psi_\alpha + 2\lambda \psi_\alpha \label{eq:wp-lin-psi} & \text{if } \mu_\alpha = \lambda \\
\Delta \psi_\alpha & \text{if } \mu_\alpha = 0 \\
\end{cases}
\end{align}
\end{subequations}
for $\alpha = 1,\dots,m$.
\end{lem}

\begin{proof}
The first equation is essentially the same as in Lemma \ref{lem:hrf-deturck-lin}, so we only consider the second equation.  The Laplacian term in the second equation is already linear.  When $\mu_\alpha = \lambda$, we use the chain rule to understand the second term:
\[ \partial_\epsilon|_{\epsilon=0} \big( -\lambda e^{-2(\tilde{\phi}_\alpha(\epsilon)-(\phi_\alpha)_0)} \big)
= 2 \lambda \psi e^{-2(\tilde\phi_\alpha(0)-(\phi_\alpha)_0)} = 2\lambda \psi. \]
The Lie derivative from the DeTurck trick term is handled essentially as in the proof of Lemma \ref{lem:hrf-deturck-lin}.
\end{proof}
  
Now assume that $(g_0,\Phi_0)$ is a fixed point of the curvature-normalized multiply-warped product Ricci-DeTurck flow with $\Phi_0$ constant and $g_0$ a strictly linearly stable $\lambda$-Einstein metric.

\begin{lem}\label{lem:wp-stable}
Let $\phi_\alpha$ be constant function.  If $\mu_\alpha = 0$, then $\phi_\alpha$ is weakly linearly stable with respect to the operator $\mbf{L}_\alpha$, and the null eigenspace is the 1-dimensional space of constant functions on $\mcl{B}$.  If $\mu_\alpha = \lambda$, then $\phi_\alpha$ is strictly linearly stable with respect to the operator $\mbf{L}_\alpha$.
\end{lem}

\begin{proof}
When $\mu_\alpha = \lambda$, we integrate by parts:
\[ (\mbf{L}_\alpha \psi,\psi) = \int_\mcl{B} (\Delta \psi + 2\lambda\psi) \psi \, d\mu_g = -\|\nabla\psi\|^2 + 2\lambda \|\psi\|^2 \leq 0, \]
since $\lambda < 0$.  We have equality exactly when $\psi$ is the zero function.
\end{proof}

We now turn to the proof of the Theorem \ref{thm:wprf-stab}.  Again, see the appendix for the statement of Simonett's theorem.  Recall that if $\mcl{V} \rightarrow \mcl{M}$ is a vector bundle, then $\mfr{h}^{r+\rho}(\mcl{V})$ denotes the completion of the vector space $C^\infty(\mcl{V})$ with respect to the $r+\rho$ little-H\"older norm, and for brevity, let $\mfr{h}^{r+\rho}(\mcl{B})$ denote the corresponding completion of $C^\infty(\mcl{B})$.

For fixed $0 < \sigma < \rho < 1$, consider the following densely and continuously embedded spaces:
\begin{align*}
\mbb{E}_0 &:=\mfr{h}^{0+\sigma}(S^2 \mcl{B}) \times 
            \big( \mfr{h}^{0+\sigma}(\mcl{B}) \big)^m \\
\cup \, & \\            
\mbb{X}_0 &:=\mfr{h}^{0+\rho}(S^2 \mcl{B}) \times             
            \big( \mfr{h}^{0+\rho}(\mcl{B}) \big)^m \\
\cup \, & \\               
\mbb{E}_1 &:=\mfr{h}^{2+\sigma}(S^2 \mcl{B}) \times 
            \big( \mfr{h}^{2+\sigma}(\mcl{B}) \big)^m \\
\cup \, & \\               
\mbb{X}_1 &:=\mfr{h}^{2+\rho}(S^2 \mcl{B}) \times 
            \big( \mfr{h}^{2+\rho}(\mcl{B}) \big)^m
\end{align*}
For fixed $1/2\leq\beta<\alpha<1$, define the continuous interpolation spaces
\begin{align*}
\mbb{X}_\beta := (\mbb{X}_0, \mbb{X}_1)_\beta,\quad
\mbb{X}_\alpha := (\mbb{X}_0, \mbb{X}_1)_\alpha.
\end{align*}
For fixed $0<\epsilon\ll 1$, let $\mbb{G}_\beta$ be the open $\epsilon$-ball around $(g_0,\phi_0)$ in $\mbb{X}_\beta$, and define $\mbb{G}_\alpha := \mbb{G}_\beta \cap \mbb{X}_\alpha$.

\begin{proof}[Proof of Theorem \ref{thm:wprf-stab}.]
Modulo the details of the H\"older space setup, this is the same as the proof of Theorem \ref{thm:hrf-stab} so we omit the details.  We note, however, that each $\alpha$ such that $\mu_\alpha=0$ gives rise to a center manifold, due to weak linear stability.  This does not happen when $\mu_\alpha = \lambda < 0$.
\end{proof} 

\begin{rem}
One may compare Theorem \ref{thm:wprf-stab} with \cite{LottSesum2011}*{Theorem 1.1}.  Those authors prove convergence of warped product Ricci flow solutions when the base has dimension two, although different techniques are involved.
\end{rem}

\begin{rem}
Following Perelman (for the Ricci flow) and List and M\"uller (for harmonic-Ricci flow), the flow \eqref{eq:wprf} is the gradient flow of a certain energy functional.  For example, when $m=1$ and $(\mcl{F}^k,h)$ is $\mu$-Einstein, given a metric $g$ on $\mcl{B}$ and functions $\phi, f : \mcl{B} \rightarrow \R$, the energy functional is
\[ \msc{F}(g,\phi,f) := \int_\mcl{B} (R - k |d\phi|^2 + k\mu e^{-2\phi} + |df|^2) e^{-f} \, dV. \]
\end{rem}


\section{Locally $\R^N$-invariant Ricci flow}\label{sec:irf}

\subsection{Setup}\label{subsec:irf-setup}

The manifolds that we will consider in this section have a special bundle structure.  Let $\mcl{B}$ be a connected, oriented manifold, and let $\mcl{E} \xrightarrow{p} \mcl{B}$ be a flat $\R^N$-vector bundle.  We consider $\mcl{M} \xrightarrow{\pi} \mcl{B}$ to be a principal $\R^N$-bundle, twisted by $\mcl{E}$.  That is, there exists a smooth map
\[ \mcl{E} \times_{\mcl{B}} \mcl{M} = \bigcup_{b \in \mcl{B}} \mcl{E}_b \times \mcl{M}_b 
\longrightarrow \mcl{M} \]
that, over each point $b \in \mcl{B}$, gives a free and transitive action that is consistent with the flat connection on $\mcl{E}$.  This means that if $\mcl{U} \subset \mcl{B}$ is such that $\mcl{E}_\mcl{U} \rightarrow \mcl{U}$ is trivializable, then $\pi^{-1}(\mcl{U})$ has a free $\R^N$ action.  Let $\mcl{M}$ have a connection $A$ such that $A|_{\pi^{-1}(U)}$ is an $\R^N$-valued connection.  If we assume that $\mcl{M}$ also has a flat connection itself, then $A$ is globally an $\R^N$-valued 1-form.

We will use this bundle structure to describe local coordinates for $\mcl{M}$.  Let $\mcl{U} \subseteq \mcl{B}$ be an open set such that $\mcl{E}_\mcl{U} \rightarrow \mcl{U}$ is trivializable and has a local section $\sigma : \mcl{U} \rightarrow \pi^{-1}(\mcl{U})$.  Additionally, let $\rho : \R^n \rightarrow \mcl{U}$ be a parametrization of $\mcl{U}$, with coordinates $x^\alpha$, and let $e_i$ be a basis for $\R^N$.  Then we obtain coordinates $(x^\alpha,x^i)$ on $\pi^{-1}(\mcl{U})$ via
\begin{align*}
\R^n \times \R^N &\longrightarrow \pi^{-1}(\mcl{U}) \\
(x^{\alpha}, x^{i}) &\longmapsto (x^{i}e_{i}) \cdot \sigma \big( \rho(x^{\alpha}) \big)
\end{align*}
where $\cdot$ denotes the free $\R^{N}$-action described above.

Let $\mathbf{g}$ be a Riemannian metric on $\mcl{M}$ such that the $\R^N$-action is a local isometry.  With respect to the
coordinates above, one may write
\begin{align}\label{eq:the-metric}
\mathbf{g} 
&= g_{\alpha\beta} \, dx^{\alpha} \, dx^{\beta} + G_{ij} (dx^{i} + A_{\alpha}^{i} \, dx^{\alpha}) (dx^{j} + A_{\beta}^{j} \, dx^{\beta}). 
\end{align}
We will write this informally as $\mathbf{g} = (g,A,G)$, where $g(b) = g_{\alpha\beta}(b) \, dx^{\alpha} \, dx^{\beta}$ is locally a Riemannian metric on $\mcl{U} \subset \mcl{B}$, $A(b) = A_{\alpha}^{i}(b) \, dx^{\alpha}$ is locally the pullback by $\sigma$ of a connection on $\pi^{-1}(\mathcal{U}) \rightarrow \mathcal{U}$, and $G(b) = G_{ij}(b) \, dx^{i} \, dx^{j}$ is an inner product on the fiber $\mcl{M}_b$.

In \cite{Lott2010}, Lott considered metrics of the form \eqref{eq:the-metric} that evolve under Ricci flow and showed that the Ricci flow equation for $(\mcl{M},\mathbf{g})$ becomes three equations (see \cite{Lott2010}*{Equation (4.10)}):
\begin{align*}
\partial_t g &= -2 \Rc + \frac{1}{2} \nabla G \otimes \nabla G + dA \otimes dA \\
\partial_t A &= -\delta dA + \langle \nabla G,dA \rangle \\
\partial_t G &= \tau_{g,\theta} G - \frac{1}{2} dA \otimes dA
\end{align*}
This is called \textit{locally $\R^N$-invariant Ricci flow}.  We clarify the shorthand notation used in these equations.  Let $\mcl{S}_N := \SL(N,\R)/\SO(N)$ be the space of symmetric positive-definite bilinear forms of fixed determinant.  The tangent space $T_G \mcl{S}_N$ at $G \in \mcl{S}_N$ consists of symmetric bilinear forms with no trace.  There is a Riemannian metric $\theta$ on $T_G \mcl{S}_N$ defined by
\begin{equation}\label{eq:S-metric}
\theta_G(X,Y) = \tr(G^{-1} X G^{-1} Y) = G^{ij} X_{jk} G^{k\ell} Y_{\ell i}.
\end{equation}
Thinking of the fiberwise inner products as a map $G : (\mcl{B},g) \rightarrow (\mcl{S}_N,\theta)$, the term $\nabla G \otimes \nabla G$ in the evolution of $g$ is defined as \eqref{eq:phi-quad}:
\[ (\nabla G \otimes \nabla G)_{\alpha\beta}
= G^{ik} G^{j\ell} \nabla_\alpha G_{ij} \nabla_\beta G_{k\ell} \]
and $\tau_{g,\theta} G$ is the harmonic map Laplacian as in \eqref{eq:tension} (see \cite{Williams2013-hrf}*{Proposition 10} for discussion of this fact).  Regarding the $\mbb{R}^N$-valued $1$-form $A$, we think of $dA \otimes dA$ as a symmetric $2$-tensor on either $\mcl{B}$ or $\mbb{R}^N$ as follows:
\begin{align*}
(dA \otimes dA)_{\alpha\beta}
&= g^{\gamma\delta} G_{ij} (dA)_{\alpha\gamma}^i (dA)_{\beta\delta}^j \\
(dA \otimes dA)_{ij} 
&= g^{\alpha\gamma} g^{\beta\delta} G_{ik} G_{j\ell} (dA)_{\alpha\beta}^k (dA)_{\gamma\delta}^{\ell}.
\end{align*}
The operator $\delta$ is the adjoint of the exterior derivative $d$, and we pair $\nabla G$ and $dA$ as follows:
\[ \langle \nabla G,dA \rangle_\alpha^i 
= g^{\beta\gamma} G^{ij} \nabla_\gamma G_{jk} (dA)_{\beta\alpha}^k. \]

\subsection{Stability}

In this section we follow the outline given in the introduction to prove Theorem \ref{thm:irf-stab}.


Following Knopf (see \cite{Knopf2009}*{Equation (1.3)}), we transform the system in the manner we've used above into one whose fixed points include $\mbf{g}=(g_0,A_0,G_0)$, where $g_0$ is Einstein, $A_0=0$, and $G_0$ is constant.  Suppose that $\overline{\mbf{g}}(\overline{t})$ is a solution of \eqref{eq:irf} for $\overline{t} \geq 0$.  For $\lambda < 0$, let $\sigma(\overline{t}) := 1-2\lambda \overline{t}$ and consider
\[ \overline{\mbf{g}}(\overline{t}) \longmapsto \mbf{g}(t), \]
where
\[ 
t := -\frac{1}{2\lambda} \log \sigma(\overline{t}), \quad
g := \sigma^{-1} \overline{g}, \quad 
A := \sigma^{-1/2} \overline{A}, \quad
G := \overline{G}.
 \]
A calculation shows that this transformation results in the modified flow
\begin{equation}\label{eq:cnirf}
\begin{aligned}
\partial_t g &= -2 \Rc + \frac{1}{2} \nabla G \otimes \nabla G + dA \otimes dA + 2 \lambda g \\
\partial_t A &= -\delta dA + \langle \nabla G,dA \rangle + \lambda A \\
\partial_t G &= \tau_{g,\theta} G - \frac{1}{2} dA \otimes dA
\end{aligned}
\end{equation}
We call this system \textit{curvature-normalized locally $\mathbb{R}^{N}$-invariant Ricci flow}.

As before, we use the DeTurck trick to make the linear \eqref{eq:cnirf} system strictly parabolic.  Fix a background metric $g_0$ on $\mcl{B}$ (which we may take to be a fixed-point metric) and define a vector field $W$ depending on $\mbf{g}(t)$ by
\[ W^{\gamma} := g^{\alpha \beta}(\Gamma_{\alpha \beta}^{\gamma}
              - {}^{g_0} \Gamma_{\alpha \beta}^{\gamma}), \qquad
(W_\flat)_{k} := (\delta A)_k	\]
for $\gamma = 1,\dots,n$ and $k =1,\dots,N$.  Let $F_{t}$ be diffeomorphisms generated by $W(t)$, with initial condition $F_0 = \id$.  The one-parameter family of metrics $F_{t}^{\ast} \mbf{g}(t)$ is the solution of 
\begin{equation}\label{eq:cnirdtf}
\begin{aligned}
\partial_t g &= -2 \Rc + \frac{1}{2} \nabla G \otimes \nabla G + dA \otimes dA + 2 \lambda g - \mcl{L}_W g\\
\partial_t A &= -\delta dA + \langle \nabla G,dA \rangle + \lambda A - \mcl{L}_W A\\
\partial_t G &= \tau_{g,\theta} G - \frac{1}{2} dA \otimes dA - \mcl{L}_W G
\end{aligned}
\end{equation}
which we call \textit{curvature-normalized $\R^N$-invariant Ricci--DeTurck flow}.

A stationary solution of (\ref{eq:cnirf}) is also a stationary solution of the curvature-normalized Ricci-DeTurck flow, and we can describe a large class of such fixed points.

\begin{lem}\label{lem:cnirf-fixed-points}
Suppose that $(\mcl{M} = \R^N \times \mcl{B},\mbf{g})$ is a twisted principal $\R^N$-bundle with locally $\R^N$-invariant metric $\mbf{g}_0=(g_0,A_0,G_0)$.  Suppose that $(\mcl{B},g_0)$ is closed and $\lambda$-Einstein, $A_0=0$, and $G_0$ is constant.  Then $\mbf{g}_0$ a stationary solution of \eqref{eq:cnirf} and the curvature-normalized locally $\mbb{R}^N$-invariant Ricci-DeTurck flow.
\end{lem}

To analyze the stability near a fixed point, we must compute the linearization of the flow.  Write $\mathbf{g}_0 = (g_0,A_0,G_0)$ for such a fixed point.  Let $\widetilde{\mathbf{g}}(\epsilon) = \big( \widetilde{g}(\epsilon),\widetilde{A}(\epsilon),\widetilde{G}(\epsilon) \big)$
be a variation of $\mathbf{g}$ such that
\begin{equation}\label{eq:inv-var}
\widetilde{\mathbf{g}}(0) = \mathbf{g}_0, \quad 
\partial_\epsilon \big|_{\epsilon=0} \widetilde{\mathbf{g}} = \mathbf{h} = (h,B,H).
\end{equation}
More explicitly,
\[ \begin{aligned}
\widetilde{g}(0) &= g_0, 
                 &\partial_\epsilon \big|_{\epsilon=0} \widetilde{g} 
								 &= h \in S^2 \mcl{B}, \\
\widetilde{A}(0) &= 0, 
                 &\partial_\epsilon \big|_{\epsilon=0} \widetilde{A} 
								 &= B \in C^\infty(T^*B \otimes \R^N), \\
\widetilde{G}(0) &= G_0, 
                 &\partial_\epsilon \big|_{\epsilon=0} \widetilde{G} 
								 &= H \in C^\infty(G_0^* T\mcl{S}_N),
\end{aligned} \]
where $\widetilde{G}(\epsilon,b) = \exp_{G_0}(\epsilon H(b))$ as in \eqref{eq:hrf-var}.

\begin{lem}
The linearization of curvature-normalized locally $\mbb{R}^N$-invariant Ricci-DeTurck flow at a fixed point $\mathbf{g}_0=(g_0,A_0,G_0)$ where $g_0$ is $\lambda$-Einstein, $A$ vanishes, and $G_0$ constant is the autonomous, self-adjoint, strictly parabolic system
\begin{subequations}
\begin{align}
\partial_t h &= \mathbf{L}_0 h := \Delta_{\ell} h + 2\lambda h \label{eq:inv-L0} \\
\partial_t B &= \mathbf{L}_1 B := \Delta_d h + \lambda B \label{eq:inv-L1} \\
\partial_t H &= \mathbf{L}_2 H := \Delta H \label{eq:inv-L2}
\end{align}
\end{subequations}
where $-\Delta_d=d\delta+\delta d$ denotes Laplace-Beltrami operator on $\mbb{R}^N$-valued $1$-forms and $\mbf{L}_2 = \Delta$ satisfies $(\Delta H)_{ij} = \Delta (H_{ij})$ in local coordinates.
\end{lem}

\begin{proof}
This is similar to the proof of Lemma \ref{lem:hrf-deturck-lin}; most of the ``cross terms'' vanish in the linearization due to $A_0$ and $G_0$ being constant.  For the second equation, the main point is that the first term and the DeTurck term combine to give the Laplace-Beltrami operator:
\[ \partial_\epsilon\big|_{\epsilon=0} \Big( 
-\big(\delta d \widetilde{A}(\epsilon) \big)_\alpha^i - \big(\mcl{L}_W \widetilde{A}(\epsilon)\big)_\alpha^i \Big)
= -(\delta d B)_\alpha^i - (d \delta B)_\alpha^i
= \Delta_d B_\alpha^i, \]
since $(\mcl{L}_W \mathbf{g})_{\alpha i} = (d \delta A)_\alpha^i$.
\end{proof}


\begin{lem}\label{lem:inv-stable}
The trivial form $A_0=0$ is strictly linearly stable with respect to the operator $\mathbf{L}_{1}$.  The constant map $G_0$ is weakly linearly stable with respect to operator $\mathbf{L}_{2}$.  Its null eigenspace is the space of constant variations $F \in C^\infty(G_0^* T\mcl{S}_N)$, whose dimension is equal to $\dim \mcl{S}_N$.
\end{lem}

\begin{proof}
This follows in the same way as Lemmas \ref{lem:hrf-stable} and \ref{lem:wp-stable}.
\end{proof}

We now turn to the proof of the the main theorem.  Again, see the appendix for the statement of Simonett's theorem.  Recall that if $\mcl{V} \rightarrow \mcl{M}$ is a vector bundle, then $\mfr{h}^{r+\rho}(\mcl{V})$ denotes the completion of the vector space $C^\infty(\mcl{V})$ with respect to the $r+\rho$ little-H\"older norm.  For fixed $0 < \sigma < \rho < 1$, consider the following densely and continuously embedded spaces:
\begin{align*}
\mbb{E}_0 &:=\mfr{h}^{0+\sigma}(S^2 \mcl{B}) \times 
						\mfr{h}^{0+\sigma}(T^*\mcl{B} \otimes \R^N) \times 
            \mfr{h}^{0+\sigma}(G_0^*T \mcl{S}_N) \\
\cup \, & \\            
\mbb{X}_0 &:=\mfr{h}^{0+\rho}(S^2 \mcl{B}) \times  
						\mfr{h}^{0+\rho}(T^* \mcl{B} \otimes \R^N)  \times             
            \mfr{h}^{0+\rho}(G_0^*T \mcl{S}_N) \\
\cup \, & \\               
\mbb{E}_1 &:=\mfr{h}^{2+\sigma}(S^2 \mcl{B}) \times 
						\mfr{h}^{2+\sigma}(T^* \mcl{B} \otimes \R^N)  \times 
            \mfr{h}^{2+\sigma}(G_0^*T \mcl{S}_N) \\
\cup \, & \\               
\mbb{X}_1 &:=\mfr{h}^{2+\rho}(S^2 \mcl{B}) \times 
						\mfr{h}^{2+\rho}(T^* \mcl{B} \otimes \R^N)  \times 
            \mfr{h}^{2+\rho}(G_0^*T \mcl{S}_N)
\end{align*}
For fixed $1/2\leq\beta<\alpha<1$, define the continuous interpolation spaces
\begin{align*}
\mbb{X}_\beta := (\mbb{X}_0, \mbb{X}_1)_\beta,\quad
\mbb{X}_\alpha := (\mbb{X}_0, \mbb{X}_1)_\alpha.
\end{align*}
For fixed $0<\epsilon\ll 1$, let $\mbb{G}_\beta$ be the open $\epsilon$-ball around $\mbf{g}_0$ in $\mbb{X}_\beta$, and define $\mbb{G}_\alpha := \mbb{G}_\beta \cap \mbb{X}_\alpha$.

\begin{proof}[Proof of Theorem \ref{thm:irf-stab}.]
Modulo the details of the H\"older space setup, this is the same as the proof of Theorem \ref{thm:hrf-stab}, so we omit the details.
\end{proof}


\section{Connection Ricci flow}\label{sec:crf}

\subsection{Setup}\label{subsec:crf-setup}

Let $(\mcl{B}^n,g)$ be a Riemannian manifold with $n\geq 3$.  Choose local coordinates $(x^i)$ Suppose that $\tau$ is a $(2,1)$-tensor on $\mcl{B}$, and consider the $(3,0)$-tensor $H$ with components $H_{ijk}=g_{k\ell} \tau_{ij}^\ell$.  We can think of $\tau$ as the torsion of a connection $\nabla$ that is compatible with $g$, and we say that $\tau$ is \textit{geometric} if $H \in \Omega^3(\mcl{B})$ and $dH=0$.  Define a $(2,0)$-tensor $\mcl{H}$ (as above) by 
\[ \mcl{H}_{ij} = g^{pq} g^{rs} H_{ipr} H_{jqs} = g^{pq} g_{rs} \tau_{ip}^r \tau_{jq}^s.  \]
The Ricci curvature of $\nabla$ is a $(2,0)$-tensor on $\mcl{B}$, but it is not symmetric.  Therefore, consider the symmetric and anti-symmetric parts, denoted by $\Rc^\otimes$ and $\Rc^\wedge$, respectively:
\[ \Rc^\otimes = {}^g\! \Rc + \frac{1}{2} \mcl{H}, \qquad
   \Rc^\wedge  = -\frac{1}{2} d^* H \]
where ${}^g\! \Rc$ is the Ricci curvature of the Levi-Civita connection of $g$ and $d^* H$ has components $(d^* H)_{ij} = -g^{\ell m} {}^g \nabla_{\ell} H_{mij}$.

Now, one can consider the evolution of a connection ${}^{g(t)}\nabla + \tau(t)$ on $\mcl{B}$ in terms of the metric $g$ and geometric torsion $\tau$ of $\nabla$,
\begin{align*}
\partial_t g &= -2 \Rc^\otimes \\
\partial_t H &= 2d \Rc^\wedge
\end{align*}
From the expressions for the symmetric and anti-symmetric parts of $\Rc$, these equations become
\begin{align*}
\partial_t g &= -2 \Rc + \frac{1}{2} \mcl{H} \\
\partial_t H &= \Delta_d H
\end{align*}
which is the flow \eqref{eq:crf} (although we will use $H$ in place of $\tau$, for clarity).  It is easy to check that the property that $\tau$ is geometric is preserved under the flow, and that the flow enjoys short-time existence and uniqueness of solutions.

We will consider this flow where $g$ is a metric and $H$ is \textit{any} closed three-form, not necessarily dual to the torsion of a connection.  Such a coupling arises in physics, for example as the renormalization group flow with $B$-field \cites{Strominger1986,OliynykEtAl2006}.

Our goal is to show that this flow is stable when $g$ is Einstein and $H=0$.  In the context of connection Ricci flow, this is stability at the Levi-Civita connection of $g$, that is, where $\tau = 0$.

\subsection{Stability}

In this section we follow the outline given in the introduction to prove Theorem \ref{thm:crf-stab}.


We transform the system into one whose fixed points include pairs $(g,H)$ with $g$ Einstein and $H$ vanishing.  Suppose that $\big(\overline{g}(\overline{t}),\overline{H}(\overline{t})\big)$ is a solution of \eqref{eq:crf} for $\overline{t} \geq 0$.  For $\lambda < 0$, let $\sigma(\overline{t}) := 1-2\lambda \overline{t}$ and consider
\[ \big( \overline{g}(\overline{t}),\overline{H}(\overline{t}) \big) \mapsto \big( g(t), H(t) \big), \]
where
\[ t := -\frac{1}{2\lambda} \log \sigma(\overline{t}), \quad
g := \sigma^{-1} \overline{g}, \quad 
H := \overline{H}.
 \]
A straightforward calculation shows that this transformation results in the modified flow
\begin{subequations}\label{eq:cncrf}
\begin{align}
\partial_t g &= -2 \Rc + \frac{1}{2} \mcl{H} + 2\lambda g \\
\partial_t H &= \Delta_d H + 2 \lambda H
\end{align}
\end{subequations}
Call this the \textit{curvature-normalized connection-Ricci flow}.

As before, we use the DeTurck trick to make the linear \eqref{eq:cncrf} system strictly parabolic.  Let $g_0$ be a background metric and define a vector field $W$ depending on $g(t)$ by
\[ W^k := g^{ij}(\Gamma_{ij}^k
              - {}^{g_0} \Gamma_{ij}^k), \]
for $k = 1,\dots,n$.  Let $F_t$ be diffeomorphisms generated by $W(t)$, with initial condition $F_0 = \id$.  The one-parameter family $\big(F_t^* g(t),F_t^* H(t)\big)$ is the solution of 
\begin{subequations}\label{eq:cncrdtf}
\begin{align}
\partial_t g &= -2 \Rc + \frac{1}{2} \mcl{H} + 2\lambda g - \mcl{L}_W g \\
\partial_t H &= \Delta_d H + 2 \lambda H - \mcl{L}_W H
\end{align}
\end{subequations}
which we call \textit{curvature-normalized connection Ricci--DeTurck flow}.  

A stationary solution of \eqref{eq:cncrf} is also a stationary solution of \eqref{eq:cncrdtf}, and we can describe a large class of such fixed points.

\begin{lem}\label{lem:crf-fixed-points}
Suppose that $(\mcl{B},g_0)$ is closed and $\lambda$-Einstein, and that $H_0=0$.  The pair $(g_0,H_0)$ a stationary solution of \eqref{eq:cncrf} and \eqref{eq:cncrdtf}.
\end{lem}

To analyze the stability near a fixed point $(g_0,H_0)$, we must compute the linearization of the flow.  Let $\big( \widetilde{g}(\epsilon),\widetilde{H}(\epsilon)\big)$ be a variation of $(g_0,H_0)$ such that
\begin{equation}\label{eq:crf-var}
\begin{aligned}
\widetilde{g}(0) &= g_0, 
                 &\partial_\epsilon \big|_{\epsilon=0} \widetilde{g} 
								 &= h \in S^2 \mcl{B}, \\
\widetilde{H}(0) &= 0, 
                 &\partial_\epsilon \big|_{\epsilon=0} \widetilde{H} 
								 &= \eta \in \Omega^3_\mrm{closed}(\mcl{B}).
\end{aligned}
\end{equation}
The next two lemmas follow as before.

\begin{lem}
The linearization of \eqref{eq:cncrdtf} at a fixed point $(g_0,H_0)$ 
where $g_0$ is $\lambda$-Einstein and $H_0=0$ is the autonomous, self-adjoint, strictly parabolic system
\begin{subequations}
\begin{align}
\partial_t h &= \mathbf{L}_2 h := \Delta_{\ell} h + 2\lambda h \label{eq:crf-L0} \\
\partial_t \eta &= \mathbf{L}_1 \eta := \Delta_{d} \eta + 2\lambda \eta \label{eq:crf-L1}
\end{align}
\end{subequations}
\end{lem}

\begin{lem}\label{lem:crf-stable}
The Levi-civita connection of $g_0$ is strictly linearly stable with respect to the operator $\mathbf{L}_{1}$.
\end{lem}


We now turn to the proof of Theorem \ref{thm:crf-stab}.  Again, see the appendix for the statment of Simonett's theorem.  Recall that if $\mcl{V} \rightarrow \mcl{B}$ is a vector bundle, then $\mfr{h}^{r+\rho}(\mcl{V})$ denotes the completion of the vector space $C^\infty(\mcl{V})$ with respect to the $r+\rho$ little-H\"older norm.  For fixed $0 < \sigma < \rho < 1$, consider the following densely and continuously embedded spaces:
\begin{align*}
\mbb{E}_0 &:=\mfr{h}^{0+\sigma}(S^2 \mcl{B}) \times 
						\mfr{h}^{0+\sigma}\big(\Omega_\mrm{closed}^3(\mcl{B})\big) \\
\cup \, & \\            
\mbb{X}_0 &:=\mfr{h}^{0+\rho}(S^2 \mcl{B}) \times  
						\mfr{h}^{0+\rho}\big(\Omega_\mrm{closed}^3(\mcl{B})\big) \\
\cup \, & \\               
\mbb{E}_1 &:=\mfr{h}^{2+\sigma}(S^2 \mcl{B}) \times 
						\mfr{h}^{2+\sigma}\big(\Omega_\mrm{closed}^3(\mcl{B})\big) \\
\cup \, & \\               
\mbb{X}_1 &:=\mfr{h}^{2+\rho}(S^2 \mcl{B}) \times 
						\mfr{h}^{2+\rho}\big(\Omega_\mrm{closed}^3(\mcl{B})\big)
\end{align*}
For fixed $1/2\leq\beta<\alpha<1$, define the continuous interpolation spaces
\begin{align*}
\mbb{X}_\beta := (\mbb{X}_0, \mbb{X}_1)_\beta,\quad
\mbb{X}_\alpha := (\mbb{X}_0, \mbb{X}_1)_\alpha.
\end{align*}
For fixed $0<\epsilon\ll 1$, let $\mbb{G}_\beta$ be the open $\epsilon$-ball around $(g_0,\tau_0)$ in $\mbb{X}_\beta$, and define $\mbb{G}_\alpha := \mbb{G}_\beta \cap \mbb{X}_\alpha$.

\begin{proof}[Proof of Theorem \ref{thm:crf-stab}.]
Modulo the details of the H\"older space setup, this is the same as the proof of Theorem \ref{thm:hrf-stab}, so we omit the details.  Note, however, that it only uses a special case of Simonett's theorem, since there are no center manifolds.
\end{proof}


\appendix
\section{Stability Theorem}\label{app:stab}

We use the following version of Simonett's Stability Theorem. Other versions of the theorem are found in \cites{GuentherIsenbergKnopf2002,Knopf2009}, while the most general version is found in \cite{Simonett1995}.

\begin{thm}
[Simonett]\label{thm:simonett}
Assume the following conditions hold:
\begin{enumerate}
\item \label{SimFirst}
$\mbb{X}_1 \hookrightarrow \mbb{X}_0$ and $\mbb{E}_1 \hookrightarrow \mbb{E}_0$ are continuous dense inclusions of Banach spaces.  For fixed $0<\beta<\alpha<1$, $\mbb{X}_\alpha$ and $\mbb{X}_\beta$ are continuous interpolation spaces corresponding to the inclusion $\mbb{X}_1 \hookrightarrow \mbb{X}_0$.

\item There is an autonomous quasilinear parabolic equation
\begin{equation}\label{eq:apqe}
\partial_\tau \widetilde{\mbf{g}}(\tau) = \mbf{Q}(\widetilde{\mbf{g}}(\tau)), 
\qquad(\tau \geq 0), 
\end{equation}
with the property that there exists a positive integer $k$ such that for all $\widehat{\mbf{g}}$ in some open set $\mbb{G}_\beta \subseteq \mbb{X}_\beta$, the domain $\mbb{D}(\mbf{L}_{\widehat{\mbf{g}}})$ of the linearization $\mbf{L}_{\widehat{\mbf{g}}}$ of $\mbf{Q}$ at $\widehat{\mbf{g}}$ contains $\mbb{X}_1$ and the map $\widehat{\mbf{g}} \mapsto \mbf{L}_{\widehat{\mbf{g}}}|_{\mbb{X}_1}$ belongs to $C^{k}(\mbb{G}_\beta, \mcl{L}(\mbb{X}_1,\mbb{X}_0))$.

\item For each $\widehat{\mbf{g}} \in \mbb{G}_\beta$, there exists an extension $\widehat{\mbf{L}}_{\widehat{\mbf{g}}}$ of $\mbf{L}_{\widehat{\mbf{g}}}$ to a domain $\widehat{\mbb{D}}(\widehat{\mbf{g}})$ that contains $\mbb{E}_{1}$ (hence is dense in
$\mbb{E}_0$).

\item \label{Sectorial} For each $\widehat{\mbf{g}} \in \mbb{G}_\alpha = \mbb{G}_\beta \cap \mbb{X}_\alpha$, $\widehat{\mbf{L}}_{\widehat{\mbf{g}}}|_{\mbb{E}_1} \in \mcl{L}(\mbb{E}_1,\mbb{E}_0)$ generates a strongly-continuous analytic semigroup on
$\mcl{L}(\mbb{E}_0,\mbb{E}_0)$. (Observe that for $\widehat{\mbf{g}}\in\mbb{G}_\alpha$, this implies that $\widehat{\mbb{D}}(\widehat{\mbf{g}})$ becomes a Banach space when equipped with the graph norm with respect to $\mbb{E}_0$.)

\item For each $\widehat{\mbf{g}} \in \mbb{G}_\alpha$, $\mbf{L}_{\widehat{\mbf{g}}}$ is the part of $\widehat{\mbf{L}}_{\widehat{\mbf{g}}}$ in $\mbb{X}_0$.\footnote{If $\mbb{X}$ is a Banach space with subspace $\mbb{Y}$ and $L : D(L) \subseteq \mbb{X} \rightarrow \mbb{X}$ is linear, then $L^{\mbb{Y}}$, \textit{the part of }$L$\textit{ in }$\mbb{Y}$, is defined by the action $L^{\mbb{Y}} : x \mapsto Lx$ on the domain $D(L^{\mbb{Y}})=\{x\in D(L):Lx\in\mbb{Y\}}$.}

\item \label{CIspaces} For each $\widehat{\mbf{g}} \in \mbb{G}_\alpha$, there
exists $\theta \in (0,1)$ such that $\mbb{X}_0 \cong (\mbb{E}_0,\widehat{\mbb{D}}(\widehat{\mbf{g}}))_{\theta}$ and $\mbb{X}_1 \cong (\mbb{E}_0,\widehat{\mbb{D}}(\widehat{\mbf{g}}))_{1+\theta}$, where $(\mbb{E}_0,\widehat{\mbb{D}}(\widehat{\mbf{g}}))_{1+\theta} = \{\mbf{g}\in\widehat{\mbb{D}}(\widehat{\mbf{g}}) : \widehat{\mbf{L}}_{\widehat{\mbf{g}}}(\mbf{g}) \in (\mbb{E}_0,\widehat{\mbb{D}}(\widehat{\mbf{g}}))_{\theta}\}$ as a set, endowed with the graph norm of $\widehat{\mbf{L}}_{\widehat{\mbf{g}}}$ with respect to $(\mbb{E}_{0},\widehat{\mbb{D}}(\widehat{\mbf{g}}))_{\theta}$.

\item \label{SimLast} $\mbb{E}_1 \hookrightarrow \mbb{X}_\beta \hookrightarrow \mbb{E}_0$ is a continuous and dense inclusion such that there exist $C>0$ and $\delta \in (0,1)$ such that for all $\eta \in \mbb{E}_1$, one has
\[ \| \eta \|_{\mbb{X}_\beta} \leq C \| \eta \|_{\mbb{E}_0}^{1-\delta} \| \eta \|_{\mbb{E}_1}^\delta. \]
\end{enumerate}

Let $\mbf{L}_{\mbf{g}}^{\mbb{C}}$ denote the complexification of the linearization $\mbf{L}_{\mbf{g}}$ of (\ref{eq:apqe}) at a stationary solution $\mbf{g}$ of (\ref{eq:apqe}
).\footnote{Note that $\mbf{L}_{\mbf{g}}$ is the operator that appears in Assumption~2.} Suppose there exists $\lambda_\mrm{s}>0$ such that the spectrum $\mbf{\sigma}$ of $\mbf{L}_{\mbf{g}}^{\mbb{C}}$ admits the decomposition $\sigma = \sigma_{\mrm{s}} \cup \{0\}$, where $0$ is an eigenvalue of finite multiplicity and $\sigma_{\mrm{s}} \subseteq \left\{z:\operatorname{Re}z \leq - \lambda_{\mrm{s}} \right\}$.  If the above Assumptions hold, then:

\begin{enumerate}
\item For each $\alpha \in [0,1]$, there is a direct-sum decomposition $\mbb{X}_\alpha = \mbb{X}_\alpha^{\mrm{s}} \oplus \mbb{X}_\alpha^{\mrm{c}}$, where $\mbb{X}_{\alpha}^{\mrm{c}}$ is the finite-dimensional algebraic eigenspace corresponding to the null eigenvalue of $\mbf{L}_{\mbf{g}}^{\mbb{C}}$.

\item For each $r \in \mbb{N}$, there exists $d_{r}>0$ such that for all $d\in(0,d_{r}]$, there exists a bounded $C^{r}$ map $\gamma_{d}^{r} : B(\mbb{X}_{1}^{\mrm{c}},\mbf{g},d) \rightarrow \mbb{X}_{1}^{\mrm{s}}$ such that $\gamma_{d}^{r}(\mbf{g}) = 0$ and $D\gamma
_{d}^{r}(\mbf{g}) = 0$.  The image of $\gamma_{d}^{r}$ lies in the closed ball $\bar{B}(\mbb{X}_{1}^{\mrm{s}},\mbf{g},d)$.  Its graph is a local $C^{r}$ center manifold $\Gamma_{\mrm{loc}}^{r} = \{(h,\gamma_{d}^{r}(h)) : h \in B(\mbb{X}_1^{\mrm{c}},\mbf{g},d)\} \subset \mbb{X}_1$ satisfying $T_{\mbf{g}}\Gamma_{\mrm{loc}}^{r} \cong \mbb{X}_1^{\mrm{c}}$.  Moreover, $\Gamma_{\mrm{loc}}^{r}$ is invariant for solutions of (\ref{eq:apqe}) as long as they remain in $B(\mbb{X}_{1}^{\mrm{c}},\mbf{g},d) \times B(\mbb{X}_1^{\mrm{s}},0,d)$.

\item Fix $\lambda \in (0,\lambda_{\mrm{s}})$. Then for each $\alpha \in (0,1)$, there exist $C>0$ and $d\in(0,d_{r}]$ such that for each initial datum $\widetilde{\mbf{g}}(0) \in B(\mbb{X}_\alpha,\mbf{g},d)$ and all times $\tau \geq 0$ such that $\widetilde{\mbf{g}}(\tau)\in B(\mbb{X}_\alpha,\mbf{g},d)$, the center manifold $\Gamma_{\mrm{loc}}^{r}$ is exponentially attractive in the stronger space $\mbb{X}_1$ in the sense that
\[ \| \pi^{\mrm{s}} \widetilde{\mbf{g}}(\tau) -  \gamma_{d}^{r}(\pi^{\mrm{c}}\widetilde{\mbf{g}}(\tau))\|_{\mbb{X}_1} 
\leq \frac{C_{\alpha}}{\tau^{1-\alpha}} e^{-\lambda\tau}
\| \pi^{\mrm{s}}\widetilde{\mbf{g}}(0) - \gamma_{d}^{r}(\pi^{\mrm{c}}\widetilde{\mbf{g}}(0)) \|_{\mbb{X}_\alpha}. \]
Here, $\widetilde{\mbf{g}}(\tau)$ is the unique solution of (\ref{eq:apqe}), while $\pi^{\mrm{s}}$ and $\pi^{\mrm{c}}$ denote the projections onto $\mbb{X}_{\alpha}^{\mrm{s}} \cong(\mbb{X}_{1}^{\mrm{s}},\mbb{X}_{0}^{\mrm{s}})_{\alpha}$ and $\mbb{X}_{\alpha}^{\mrm{c}}$, respectively.
\end{enumerate}
\end{thm}




\end{document}